\documentclass[letterpaper,10pt]{article}

\usepackage[margin=0.5in]{geometry}

\usepackage{fullpage}
\usepackage{enumerate}
\usepackage{authblk}
\usepackage{verbatim}
\usepackage{section, amsthm, textcase, setspace, amssymb, amscd, lineno, amsmath, amsfonts, latexsym, fancyhdr, longtable, ulem, mathtools}
\usepackage{epsfig, graphicx, pstricks,pst-grad,pst-text,tikz,colortbl}
\usepackage{epsf}
\usepackage{graphicx, color}
\usepackage{float}
\usepackage[rflt]{floatflt}
\usepackage{amsfonts}
\usepackage{latexsym,enumitem}
\usetikzlibrary{fit,matrix,positioning}
\usepackage{pdflscape}
\usetikzlibrary{decorations.pathreplacing}
\usepackage[most]{tcolorbox}
\usepackage{lipsum}
\usetikzlibrary{decorations.markings}

\usepackage{mathrsfs}

\newcommand{\ldownarrow}{\Big\downarrow}

\newcommand{\precdot}{\prec\mathrel{\mkern-5mu}\mathrel{\cdot}}

\newtheorem{theorem}{Theorem}[section]
\newtheorem{lemma}[theorem]{Lemma}
\newtheorem{corollary}[theorem]{Corollary}

\newtheorem{conj}[theorem]{Conjecture}

\newtheorem{definition}[theorem]{Definition}

\newtheorem{example}[theorem]{Example}

\newtheorem{prop}[theorem]{Proposition}
\newtheorem*{theorem*}{Theorem}
\newtheorem{remark}[theorem]{Remark}

\newcommand{\rn}[1]{{\color{red} #1}}

\usepackage{xargs}
\usepackage[colorinlistoftodos,prependcaption,textsize=tiny,linecolor=red,backgroundcolor=red!25,bordercolor=red]{todonotes}
\usepackage[backend=bibtex]{biblatex}
\begin{document}

\title{On a combinatorial puzzle arising from the theory of Lascoux polynomials}

\author[*]{Kelsey Hanser}

\author[*]{Nicholas Mayers}

\affil[*]{Department of Mathematics, North Carolina State University, Raleigh, NC, 27605}

\maketitle

\bigskip
\begin{abstract} 
\noindent
Lascoux polynomials are a class of nonhomogeneous polynomials which form a basis of the full polynomial ring. Recently, Pan and Yu showed that Lascoux polynomials can be defined as generating polynomials for certain collections of diagrams consisting of unit cells arranged in the first quadrant generated from an associated ``key diagram" by applying sequences of ``$K$-Kohnert moves". Within diagrams generated in this manner, certain cells are designated as special and referred to as ``ghost cells". Given a fixed Lascoux polynomial, Pan and Yu established a combinatorial algorithm in terms of ``snow diagrams" for computing the maximum number of ghost cells occurring in a diagram defining a monomial of the given polynomial; having this value allows for one to determine the total degree of the given Lascoux polynomial. In this paper, we study the combinatorial puzzle which arises when one replaces key diagrams by arbitrary diagrams in the definition of Lascoux polynomials. Specifically, given an arbitrary diagram, we consider the question of determining the maximum number of ghost cells contained within a diagram among those formed from our given initial one by applying sequences of $K$-Kohnert moves. In this regard, we establish means of computing the aforementioned max ghost cell value for various families of diagrams as well as for diagrams in general when one takes a greedy approach.
\end{abstract}


\section{Introduction}

Lascoux polynomials, introduced in \textbf{\cite{Lascoux}}, are a class of nonhomogeneous polynomials, indexed by weak compositions, which form a basis of the full polynomial ring. Such polynomials are the $K$-theoretic analogues of key polynomials and are generalizations of Grothendieck polynomials. Recently, in \textbf{\cite{Pan1}}, it was shown that the monomials of a given Lascoux polynomial encode diagrams belonging to a collection formed by starting from an associated ``key diagram" and applying sequences of what have been called ``$K$-Kohnert moves"; such definitions for families of polynomials in terms of diagrams and certain moves are not new -- originating in the thesis of Kohnert (\textbf{\cite{Kohnert}}, 1990) and investigated further by numerous other authors (\textbf{\cite{KP3,AssafSchu,KP2,KP1,Pan1,Pan2,Winkel2,Winkel1}}). Here, we study a combinatorial puzzle arising from this definition for Lascoux polynomials.

In order to describe the combinatorial puzzle of interest, we first outline the definition for Lascoux polynomials discussed above (complete details can be found in Section~\ref{sec:prelim}). Given a weak composition $\alpha\in\mathbb{Z}^n_{\ge 0}$, we denote the corresponding Lascoux polynomial as $\mathfrak{L}_\alpha$ and associate to $\alpha$ a diagram $\mathbb{D}(\alpha)$ consisting of finitely many cells arranged into the first quadrant. See Figure~\ref{fig:intro} (a) for $\mathbb{D}(\alpha)$ with $\alpha=(0,1,2,2)$. From $\mathbb{D}(\alpha)$ we form a finite collection of diagrams, denoted by $KKD(\mathbb{D}(\alpha))$, consisting of $\mathbb{D}(\alpha)$ along with all those diagrams that can be formed from $\mathbb{D}(\alpha)$ by applying sequences of two types of moves: Kohnert and ghost moves. Briefly, Kohnert moves, when nontrivial, cause the rightmost cell of a given row to descend to the highest empty position below and in the same column. Similarly, ghost moves, when nontrivial, cause the rightmost cell of a given row to descend to the highest empty position below and in the same column, leaving a special ``ghost" cell in its place. The ghost cells introduced by ghost moves place restrictions on the effects of Kohnert and ghost moves. In particular, ghost cells are fixed by both types of move and prevent cells located strictly above and in the same column from moving to positions strictly below. In Figure~\ref{fig:intro} we illustrate: (a) $\mathbb{D}(\alpha)$ for $\alpha=(0,1,2,2)$; (b) the diagram obtained from $\mathbb{D}(\alpha)$ by applying a Kohnert move at row 3; and (c) the diagram obtained from $\mathbb{D}(\alpha)$ by applying a ghost move at row 3, where the ghost cell is decorated by an $\bigtimes$ and the shaded cells are those that are now fixed by both types of moves as a consequence of the ghost cell.

\begin{figure}[H]
    \centering
    $$\scalebox{0.9}{\begin{tikzpicture}[scale=0.6]
  \node at (0.5, 1.5) {$\cdot$};
  \node at (0.5, 2.5) {$\cdot$};
  \node at (0.5, 3.5) {$\cdot$};
  \node at (1.5, 2.5) {$\cdot$};
  \node at (1.5, 3.5) {$\cdot$};
  \draw (0,4.5)--(0,0)--(2.5,0);
  \draw (0,1)--(1,1)--(1,4)--(0,4);
  \draw (0,2)--(1,2);
  \draw (0,3)--(1,3);
  \draw (1,2)--(2,2)--(2,4)--(1,4);
  \draw (1,3)--(2,3);
  \node at (1.25, -1) {\large $(a)$};
\end{tikzpicture}}\quad\quad\quad\quad\quad\quad \scalebox{0.9}{\begin{tikzpicture}[scale=0.6]
  \node at (0.5, 1.5) {$\cdot$};
  \node at (0.5, 2.5) {$\cdot$};
  \node at (0.5, 3.5) {$\cdot$};
  \node at (1.5, 1.5) {$\cdot$};
  \node at (1.5, 3.5) {$\cdot$};
  \draw (0,4.5)--(0,0)--(2.5,0);
  \draw (0,1)--(1,1)--(1,4)--(0,4);
  \draw (0,2)--(1,2);
  \draw (0,3)--(1,3);
  \draw (1,1)--(2,1)--(2,2)--(1,2);
  \draw (1,3)--(2,3)--(2,4)--(1,4);
  \node at (1.25, -1) {\large $(b)$};
\end{tikzpicture}}\quad\quad\quad\quad\quad\quad\scalebox{0.9}{\begin{tikzpicture}[scale=0.6]
  \node at (0.5, 1.5) {$\cdot$};
  \node at (0.5, 2.5) {$\cdot$};
  \node at (0.5, 3.5) {$\cdot$};
  \node at (1.5, 1.5) {$\cdot$};
  \node at (1.5, 3.5) {$\cdot$};
  \node at (1.5, 2.5) {$\bigtimes$};
  \draw (0,4.5)--(0,0)--(2.5,0);
  \draw (0,1)--(1,1)--(1,4)--(0,4);
  \draw (0,2)--(1,2);
  \draw (0,3)--(1,3);
  \draw (1,2)--(2,2)--(2,4)--(1,4);
  \draw (1,3)--(2,3);
  \draw (1,1)--(2,1)--(2,2);
  \filldraw[fill=lightgray,fill opacity=0.4] (0,2) rectangle (1,3);
  \filldraw[fill=lightgray,fill opacity=0.4] (1,3) rectangle (2,4);
  \node at (1.25, -1) {\large $(c)$};
\end{tikzpicture}}$$
    \caption{Diagrams and moves}
    \label{fig:intro}
\end{figure}
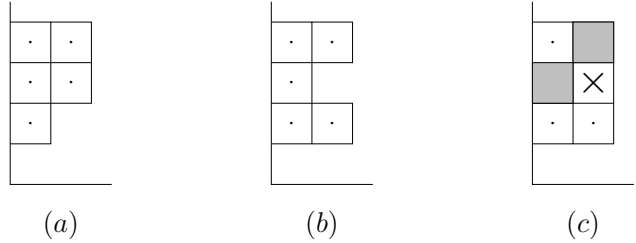

\noindent
Now, it was conjectured in \textbf{\cite{KKohnert}} and, subsequently, proven in \textbf{\cite{Pan1}} that $$\mathfrak{L}_\alpha=\sum_{D\in KKD(\mathbb{D}(\alpha))}\mathrm{wt}(D),$$ where $\mathrm{wt}(D)=(-1)^gx_1^{r_1}x_2^{r_2}\cdots x_n^{r_n}$ with $g$ the total number of ghost cells contained in $D$ and $r_i$ the number of cells, both ghost and non-ghost, contained in row $i$ of $D$. 

With the definition for Lascoux polynomials outlined above in hand, we can now define our combinatorial puzzle of interest. Given an arbitrary diagram $D$, determine a sequence of Kohnert and ghost moves which, when applied to $D$, results in a diagram containing the maximum possible number of ghost cells. Our main objective here is to establish means by which one can determine when the puzzle has been solved; that is, given an arbitrary diagram $D$, we aim to establish means of computing the maximum number of ghost cells contained in a diagram among those formed from $D$ by applying sequences of Kohnert and ghost moves, denoting this value by $\mathrm{MaxG}(D)$. We are not the first to consider this problem. In \textbf{\cite{Pan2}}, Pan and Yu establish a combinatorial algorithm for the computation of $\mathrm{MaxG}(D)$ when $D=\mathbb{D}(\alpha)$ for a weak composition $\alpha$. Their algorithm consists of decorating the diagram $\mathbb{D}(\alpha)$ with ``dark clouds" and ``snowflakes", with the value $\mathrm{MaxG}(\mathbb{D}(\alpha))$ corresponding to the number of snowflakes in the decorated diagram. 

For our contributions, we show that either the algorithm of Pan and Yu given in \textbf{\cite{Pan2}} or a slight modification applies to a larger collection of diagrams, including the ``skew" and special cases of ``lock" diagrams of \textbf{\cite{KP1}}; here, the slight modification is formed by removing certain snowflakes from the decorated diagram of \textbf{\cite{Pan2}}. Included in the study of lock diagrams is an application of a promising approach via labeling for identifying methods of computing $\mathrm{MaxG}(D)$ for families of diagrams $D$. In addition to determining means of computing $\mathrm{MaxG}(D)$ in special cases, we establish an algorithm in the spirit of that found in \textbf{\cite{Pan2}} which applies to computing the value analogous to $\mathrm{MaxG}(D)$ when taking a greedy approach for an arbitrary diagram $D$.

The remainder of the paper is organized as follows. In Section~\ref{sec:prelim}, we cover the necessary preliminaries to define and study our combinatorial puzzle. Following this, in Section~\ref{sec:extJP}, we extend a main result of \textbf{\cite{Pan2}}, showing that either the algorithm introduced in \textbf{\cite{Pan2}} for computing $\mathrm{MaxG}(\mathbb{D}(\alpha))$ for weak compositions $\alpha$ or a slight modification can be applied to various other families of diagrams. Also included in Section~\ref{sec:extJP} is a discussion and application of a promising approach to establishing means for computing $\mathrm{MaxG}(D)$ for families of diagrams $D$. In Section~\ref{sec:greedy}, we determine the limits of taking a greedy approach to the puzzle, applying only ghost moves. Finally, in Section~\ref{sec:epi}, we discuss directions for future research.

\section{Preliminaries}\label{sec:prelim}

In this section, we cover the requesit preliminaries to define our combinatorial puzzle of interest as well as discuss known results. Ongoing, for $n\in\mathbb{N}$, we let $[n]=\{1,2,\hdots,n\}$.

As mentioned in the introduction, we will be interested in applying certain moves to ``diagrams". In this paper, a \textbf{diagram} is an array of finitely many cells in $\mathbb{N}\times\mathbb{N}$, where some of the cells may be decorated with an $\bigtimes$ and called \textbf{ghost cells}. Example diagrams are illustrated in Figure~\ref{fig:diagram} (a) and (b) below. Such decorated diagrams can be defined by the set of row/column coordinates of the cells defining it, where non-ghost cells are denoted by ordered pairs of the form $(r,c)$ and ghost cells by ordered pairs of the form $\langle r,c\rangle$. Consequently, if a diagram $D$ contains a non-ghost (resp., ghost) cell in position $(r,c)$, then we write $(r,c)\in D$ (resp., $\langle r,c\rangle\in D$); otherwise, we write $(r,c)\notin D$ (resp., $\langle r,c\rangle\notin D$).

\begin{example}
    The diagrams $$D_1=\{(1,3),(2,1),(2,2),(3,2)\}\quad\quad\text{and}\quad\quad D_2=\{\langle1,3\rangle,\langle2,1\rangle,(2,2),(3,2)\}$$ are illustrated in Figures~\ref{fig:diagram} \textup{(a)} and \textup{(b)}, respectively.
    
    \begin{figure}[H]
    \centering
    $$\scalebox{0.9}{\begin{tikzpicture}[scale=0.6]
  \node at (0.5, 2.5) {$\cdot$};
  \node at (0.5, 1.5) {$\cdot$};
  \node at (1.5, 1.5) {$\cdot$};
  \node at (2.5, 0.5) {$\cdot$};
  \draw (0,3.5)--(0,0)--(3.5,0);
  \draw (0,3)--(1,3)--(1,2)--(0,2)--(0,3);
  \draw (0,2)--(1,2)--(1,1)--(0,1)--(0,2);
  \draw (1,2)--(2,2)--(2,1)--(1,1)--(1,2);
  \draw (0,1)--(1,1);
  \draw (2,0)--(2,1)--(3,1)--(3,0);
  \node at (1.75, -1) {\large $(a)$};
\end{tikzpicture}}\quad\quad\quad\quad\quad\quad\quad\quad \scalebox{0.9}{\begin{tikzpicture}[scale=0.6]
  \node at (0.5, 2.5) {$\cdot$};
  \node at (0.5, 1.5) {$\bigtimes$};
  \node at (1.5, 1.5) {$\cdot$};
  \node at (2.5, 0.5) {$\bigtimes$};
  \draw (0,3.5)--(0,0)--(3.5,0);
  \draw (0,3)--(1,3)--(1,2)--(0,2)--(0,3);
  \draw (0,2)--(1,2)--(1,1)--(0,1)--(0,2);
  \draw (1,2)--(2,2)--(2,1)--(1,1)--(1,2);
  \draw (0,1)--(1,1);
  \draw (2,0)--(2,1)--(3,1)--(3,0);
  \node at (1.75, -1) {\large $(b)$};
\end{tikzpicture}}$$
    \caption{Diagram}
    \label{fig:diagram}
\end{figure}
\end{example}

To each nonempty row of a diagram we can apply what is called a ``$K$-Kohnert move" defined as follows. Given a diagram $D$ and a nonempty row $r$ of $D$, to apply a $K$-Kohnert move at row $r$ of $D$, we first find $(r,c)$ or $\langle r,c\rangle\in D$ with $c$ maximal, i.e., the rightmost cell in row $r$ of $D$. If 
\begin{itemize}
    \item $\langle r,c\rangle\in D$ is the rightmost cell in row $r$ of $D$, i.e., the rightmost cell in row $r$ of $D$ is a ghost cell,
    \item there exists no $\widehat{r}<r$ such that $(\widehat{r},c)\notin D$, i.e., there are no empty positions below the rightmost cell in row $r$ of $D$, or
    \item there exists $\widehat{r}<r^*<r$ such that $(\widehat{r},c)\notin D$, $\langle r^*,c\rangle\in D$, and $(\tilde{r},c)$ or $\langle\tilde{r},c\rangle\in D$ for $\widehat{r}<\tilde{r}\neq r^*<r$, i.e., there exists a ghost cell between the rightmost cell in row $r$ of $D$ and the highest empty position below,
\end{itemize}
then the $K$-Kohnert move does nothing; otherwise, there are two choices: letting $\widehat{r}<r$ be maximal such that $(\widehat{r},c)\notin D$, either
\begin{itemize}
    \item[(1)] $D$ becomes $(D\backslash (r,c))\cup (\widehat{r},c)$, i.e., the rightmost cell in row $r$ of $D$ moves to the highest empty position below or
    \item[(2)] $D$ becomes $(D\backslash (r,c))\cup \{(\widehat{r},c),\langle r,c\rangle\}$, i.e., the rightmost cell in row $r$ of $D$ moves to the highest empty position below and leaves a ghost cell in its original position.
\end{itemize}
$K$-Kohnert moves of the form (1) are called \textbf{Kohnert moves}, while those of the form (2) are called \textbf{ghost moves}. We denote the diagram formed by applying a Kohnert (resp., ghost) move to a diagram $D$ at row $r$ by $\mathcal{K}(D,r)$ (resp., $\mathcal{G}(D,r)$). To aid in expressing the effect of applying a $K$-Kohnert move, we make use of the following notation. If applying a Kohnert move at row $r$ of $D$ causes the cell in position $(r,c)$ to move down to position $(\widehat{r},c)$, forming the diagram $\widehat{D}$, then we write $$\widehat{D}=\mathcal{K}(D,r)=D\ldownarrow^{(r,c)}_{(\widehat{r},c)}$$
and refer to the cell $(r,c)$ of $D$ as \textbf{movable}. Similarly, if applying a ghost move at row $r$ of $D$ causes the cell in position $(r,c)\in D$ to move down to position $(\widehat{r},c)$ in forming the diagram $\widehat{D}$, then we write $$\widehat{D}=\mathcal{G}(D,r)=D\ldownarrow^{(r,c)}_{(\widehat{r},c)}\cup\{\langle r,c\rangle\}.$$ 

\begin{example}
    Let $D$ be the diagram illustrated in Figure~\ref{fig:diagram} \textup{(a)}. The diagrams $$D_1=\mathcal{K}(D,3)=D\ldownarrow^{(3,1)}_{(1,1)}\quad\quad\text{and}\quad\quad D_2=\mathcal{G}(D,3)=D\ldownarrow^{(3,1)}_{(1,1)}\cup\{\langle 3,1\rangle\}$$ are illustrated in Figures~\ref{fig:moves} \textup{(a)} and \textup{(b)}, respectively. Note that $\mathcal{K}(D,1)=\mathcal{G}(D,1)=D$ and $\mathcal{K}(D_2,3)=\mathcal{G}(D_2,3)=D_2$.
    \begin{figure}[H]
    \centering
    $$\scalebox{0.9}{\begin{tikzpicture}[scale=0.6]
  \node at (0.5, 1.5) {$\cdot$};
  \node at (1.5, 1.5) {$\cdot$};
  \node at (2.5, 0.5) {$\cdot$};
  \node at (0.5, 0.5) {$\cdot$};
  \draw (0,3.5)--(0,0)--(3.5,0);
  \draw (0,2)--(1,2)--(1,1)--(0,1)--(0,2);
  \draw (0,1)--(1,1)--(1,0)--(0,0)--(0,1);
  \draw (1,2)--(2,2)--(2,1)--(1,1)--(1,2);
  \draw (0,1)--(1,1);
  \draw (2,0)--(2,1)--(3,1)--(3,0);
  \node at (2,-1) {(a)};
\end{tikzpicture}\quad\quad\quad\quad\quad\quad\quad\quad\begin{tikzpicture}[scale=0.6]
  \node at (0.5, 2.5) {$\bigtimes$};
  \node at (0.5, 1.5) {$\cdot$};
  \node at (1.5, 1.5) {$\cdot$};
  \node at (2.5, 0.5) {$\cdot$};
  \node at (0.5, 0.5) {$\cdot$};
  \draw (0,3.5)--(0,0)--(3.5,0);
  \draw (0,3)--(1,3)--(1,2)--(0,2)--(0,3);
  \draw (0,2)--(1,2)--(1,1)--(0,1)--(0,2);
  \draw (0,1)--(1,1)--(1,0)--(0,0)--(0,1);
  \draw (1,2)--(2,2)--(2,1)--(1,1)--(1,2);
  \draw (0,1)--(1,1);
  \draw (2,0)--(2,1)--(3,1)--(3,0);
  \node at (2,-1) {(b)};
\end{tikzpicture}}$$
    \caption{$K$-Kohnert moves}
    \label{fig:moves}
\end{figure}
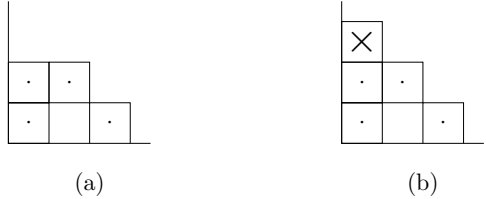
\end{example}

For a diagram $D$, we let $KKD(D)$ (resp., $KD(D)$) denote the collection of diagrams consisting of $D$ along with all those diagrams which can be formed from $D$ by applying sequences of $K$-Kohnert (resp., Kohnert) moves. Moreover, we define $$G(D)=\{\langle r,c\rangle\in D~|~r,c\in\mathbb{N}\},\quad \mathrm{MaxG}(D)=\max\{|G(\tilde{D})|~|~\tilde{D}\in KKD(D)\},$$ and $$R(D)=\{(r,c)\in D~|~(r,\tilde{c}),\langle r,\tilde{c}\rangle \notin D~\text{for}~\tilde{c}>c\};$$ that is, we denote by $G(D)$ the collection of ghost cells in $D$, $\mathrm{MaxG}(D)$ the maximum number of ghost cells contained within a diagram of $KKD(D)$, and $R(D)$ the collection of rightmost cells in $D$.

\begin{remark}\label{rem:Kohnert}
    The notion of Kohnert move was introduced in \textup{(\textbf{\cite{Kohnert}}}, 1990\textup) where A. Kohnert showed that Demazure characters \textup(a.k.a. key polynomials\textup) can be defined as generating polynomials for collections of diagrams of the form $KD(D)$. Moreover, Kohnert conjectured that such a definition could be given for Schubert polynomials as well; this conjecture has since been proven by multiple authors \textup(see \textup{\textbf{\cite{AssafSchu,Winkel2,Winkel1}}}\textup). Motivated by these developments, given a diagram $D$ containing no ghost cells, the authors of \textup{\textbf{\cite{KP1}}} define the Kohnert polynomial of $D$ as $$\mathfrak{K}_D=\sum_{\tilde{D}\in KD(D)}\mathrm{wt}(\tilde{D}),$$ where $\mathrm{wt}(D)=\prod_{r\ge 1}x_r^{|\{c~|~(r,c)\in D\}|}$; such polynomials have been of recent interest \textup{(see \textbf{\cite{KP3,KP2,KP1}})}. In \textup{\textbf{\cite{KKohnert}}}, it was conjectured that, analogous to key polynomials, Lascoux polynomials could be defined as generating polynomials for collections of diagrams of the form $KKD(D)$; this conjecture was established in \textup{\textbf{\cite{Pan1}}} and is discussed below.
\end{remark}

Using the notions defined above, as noted in the introduction, one can naturally define a combinatorial puzzle as follows. Given a diagram $D$, determine a sequence of $K$-Kohnert moves which, when applied to $D$, results in a diagram $T\in KKD(D)$ containing the maximum possible number of ghost cells, i.e., $|G(T)|=\mathrm{MaxG}(D)$. In this paper, we are interested in establishing means by which one can determine if they have solved such puzzle; that is, means of computing $\mathrm{MaxG}(D)$ for an arbitrary diagram $D$. We are not the first to consider this question. In \textbf{\cite{Pan2}}, the authors establish means of computing $\mathrm{MaxG}(D)$ when $D$ is a ``key diagram". Below we outline this result as well as discuss the motivation for the work in \textbf{\cite{Pan2}} which came from the theory of Lascoux polynomials.

Recall from the introduction that one can give a definition for Lascoux polynomials in terms of diagrams and $K$-Kohnert moves. To start, each Lascoux polynomial can be associated with a weak composition $\alpha=(\alpha_1,\hdots,\alpha_n)\in\mathbb{Z}_{\ge 0}^n$; we denote the Lascoux polynomial associated with the weak composition $\alpha$ by $\mathfrak{L}_\alpha$. Then, defining the \textbf{key diagram} associated with $\alpha$ by $$\mathbb{D}(\alpha)=\{(i,j)~|~1\le i\le n,~1\le j\le \alpha_i\}$$ (see Example~\ref{ex:kr}), it is shown in \textbf{\cite{Pan1}} that the monomials of $\mathfrak{L}_\alpha$ encode the diagrams contained in $KKD(\mathbb{D}(\alpha))$. In particular, encoding a diagram $D$ as the monomial $$\mathrm{wt}(D)=\prod_{r\ge 1}(-1)^{|G(D)|}x_r^{|\{c~|~(r,c)~\text{or}~\langle r,c\rangle\in D\}|},$$ we have the following.

\begin{theorem}[Theorem 2, \textup{\textbf{\cite{Pan1}}}]\label{thm:Lascoux}
    For a weak composition $\alpha$, $$\mathfrak{L}_\alpha=\sum_{D\in KKD(\mathbb{D}(\alpha))}\mathrm{wt}(D).$$
\end{theorem}

\begin{example}\label{ex:kr}
    Let $\alpha=(0,3,4,2,3)$. The diagram $D_1=\mathbb{D}(\alpha)$ is illustrated in Figure~\ref{fig:kr} $(a)$. This diagram contributes the term $\mathrm{wt}(D_1)=x_2^3x_3^4x_5^2x_6^3$ to $\mathfrak{L}_\alpha$, while the diagram $D_2\in KKD(D_1)$ illustrated in Figure~\ref{fig:kr} $(b)$ contributes the term $\mathrm{wt}(D_2)=(-1)^6x_1^4x_2^4x_3^4x_5^3x_6^3$. As we show below, $D_2$ contains the maximum possible number of ghost cells among diagrams contained in $KKD(D_1)$, i.e., $|G(D_2)|=6=\mathrm{MaxG}(D)$. As a consequence, the monomial corresponding to $D_2$ in $\mathfrak{L}_\alpha$ has the highest possible total degree among monomials of $\mathfrak{L}_\alpha$.
    \begin{figure}[H]
        \centering
        $$\scalebox{0.9}{\begin{tikzpicture}[scale=0.6]
        \node at (0.5, 1.5) {$\cdot$};
        \node at (1.5, 1.5) {$\cdot$};
        \node at (2.5, 1.5) {$\cdot$};
  \node at (0.5, 2.5) {$\cdot$};
  \node at (1.5, 2.5) {$\cdot$};
  \node at (2.5, 2.5) {$\cdot$};
  \node at (3.5, 2.5) {$\cdot$};
  \node at (0.5, 3.5) {$\cdot$};
  \node at (1.5, 3.5) {$\cdot$};
  \node at (0.5, 4.5) {$\cdot$};
  \node at (1.5, 4.5) {$\cdot$};
  \node at (2.5, 4.5) {$\cdot$};
  \draw (0,5.5)--(0,0)--(4.5,0);
  \draw (0,1)--(3,1)--(3,2)--(0,2);
  \draw (1,1)--(1,5);
  \draw (2,1)--(2,2);
  \draw (0,3)--(1,3);
  \draw (0,4)--(2,4)--(2,3)--(1,3);
  \draw (0,5)--(3,5)--(3,4)--(2,4);
  \draw (2,4)--(2,5);
  \draw (2,3)--(4,3)--(4,2)--(3,2);
  \draw (2,2)--(2,3);
  \draw (3,2)--(3,3);
  \node at (2, -1) {\large $(a)$};
\end{tikzpicture}}\quad\quad\quad\quad\quad\quad\quad\quad\scalebox{0.9}{\begin{tikzpicture}[scale=0.6]
        \node at (0.5, 1.5) {$\bigtimes$};
        \node at (1.5, 1.5) {$\cdot$};
        \node at (2.5, 1.5) {$\cdot$};
        \node at (3.5, 1.5) {$\bigtimes$};
  \node at (0.5, 2.5) {$\cdot$};
  \node at (1.5, 2.5) {$\cdot$};
  \node at (2.5, 2.5) {$\cdot$};
  \node at (3.5, 2.5) {$\bigtimes$};
  \node at (0.5, 3.5) {$\cdot$};
  \node at (1.5, 3.5) {$\bigtimes$};
  \node at (2.5, 3.5) {$\bigtimes$};
  \node at (0.5, 4.5) {$\cdot$};
  \node at (1.5, 4.5) {$\cdot$};
  \node at (2.5, 4.5) {$\bigtimes$};
  \node at (0.5, 0.5) {$\cdot$};
  \node at (1.5, 0.5) {$\cdot$};
  \node at (2.5, 0.5) {$\cdot$};
  \node at (3.5, 0.5) {$\cdot$};
  \draw (0,5.5)--(0,0)--(4.5,0);
  \draw (0,1)--(3,1)--(3,2)--(0,2);
  \draw (1,1)--(1,5);
  \draw (2,1)--(2,2);
  \draw (0,3)--(1,3);
  \draw (0,4)--(2,4)--(2,3)--(1,3);
  \draw (0,5)--(3,5)--(3,4)--(2,4);
  \draw (2,4)--(2,5);
  \draw (2,3)--(4,3)--(4,2)--(3,2);
  \draw (4,0)--(4,1)--(3,1)--(3,0);
  \draw (2,1)--(2,0);
  \draw (1,1)--(1,0);
  \draw (4,1)--(4,2);
  \draw (3,3)--(3,4);
  \draw (2,2)--(2,3);
  \draw (3,2)--(3,3);
  \node at (2, -1) {\large $(b)$};
\end{tikzpicture}}$$
        \caption{(a) Key diagram $D_1$ and (b) element of $KKD(D_1)$}
        \label{fig:kr}
    \end{figure}
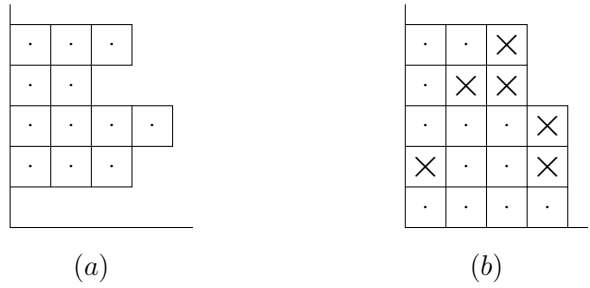
\end{example}

Considering the definition of $\mathfrak{L}_\alpha$ for $\alpha=(\alpha_1,\hdots,\alpha_n)\in\mathbb{Z}_{\ge 0}^n$ provided by Theorem~\ref{thm:Lascoux}, we see that the minimal degree of a monomial in $\mathfrak{L}_\alpha$ is given by $\sum_{i=1}^n\alpha_i$, while the maximal degree of a monomial in $\mathfrak{L}_\alpha$, i.e., the total degree of $\mathfrak{L}_\alpha$, is given by $\mathrm{MaxG}(\mathbb{D}(\alpha))+\sum_{i=1}^n\alpha_i$. Consequently, one can compute the total degree of $\mathfrak{L}_\alpha$ using the values $\mathrm{MaxG}(\mathbb{D}(\alpha))$ and $\sum_{i=1}^n\alpha_i$. In \textbf{\cite{Pan2}}, the authors establish an algorithm for computing $\mathrm{MaxG}(\mathbb{D}(\alpha))$ in terms of counting ``snowflakes" in associated ``snow diagrams". The algorithm takes as input the key diagram $\mathbb{D}(\alpha)$ and outputs a decorated diagram called a ``snow diagram" denoted $\mathsf{snow}(\mathbb{D}(\alpha))$. Given a diagram $D$ containing no ghost cells, we construct the corresponding snow diagram $\mathsf{snow}(D)$ as follows.
\begin{enumerate}
    \item Working from top to bottom, in each row mark the right-most cell which has no marked cells above it and in the same column; the marked cells are referred to as \textbf{dark clouds} and the collection of positions of dark clouds is denoted $\mathsf{dark}(D)$.
    \item Fill all empty positions below dark clouds with a snowflake $\ast$.
\end{enumerate}
In Figure~\ref{fig:krsnow}, we illustrate the snow diagram associated with the key diagram of Example~\ref{ex:kr}. Given a diagram $D$, let $\text{sf}(D)$ denote the number of snowflakes in $\mathsf{snow}(D)$.

\begin{theorem}[Theorems 1.1 and 1.2, \textup{\textbf{\cite{Pan2}}}]\label{thm:JP}
    If $D$ is a key diagram, then $\mathrm{MaxG}(D)=\textup{sf}(D)$.
\end{theorem}

\begin{example}\label{ex:snowd}
Let $\alpha=(0,3,4,2,3)$. The snow diagram $\mathsf{snow}(\mathbb{D}(\alpha))$ is illustrated in Figure~\ref{fig:krsnow}. Considering Theorem~\ref{thm:JP}, $\mathrm{MaxG}(\mathbb{D}(\alpha))=\textup{sf}(\mathbb{D}(\alpha))=6$. Thus, the total degree of $\mathfrak{L}_\alpha$ is equal to $6+\sum_{i=1}^5\alpha_i=6+12=18$.
    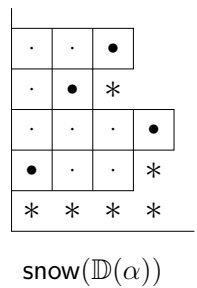
\begin{figure}[H]
        \centering
        $$\scalebox{0.9}{\begin{tikzpicture}[scale=0.6]
        \node at (0.5, 1.5) {$\bullet$};
        \node at (0.5, 0.5) {\scalebox{1.5}{$\ast$}};
        \node at (1.5, 0.5) {\scalebox{1.5}{$\ast$}};
        \node at (2.5, 0.5) {\scalebox{1.5}{$\ast$}};
        \node at (3.5, 0.5) {\scalebox{1.5}{$\ast$}};
        \node at (3.5, 1.5) {\scalebox{1.5}{$\ast$}};
        \node at (2.5, 3.5) {\scalebox{1.5}{$\ast$}};
        \node at (1.5, 1.5) {$\cdot$};
        \node at (2.5, 1.5) {$\cdot$};
  \node at (0.5, 2.5) {$\cdot$};
  \node at (1.5, 2.5) {$\cdot$};
  \node at (2.5, 2.5) {$\cdot$};
  \node at (3.5, 2.5) {$\bullet$};
  \node at (0.5, 3.5) {$\cdot$};
  \node at (1.5, 3.5) {$\bullet$};
  \node at (0.5, 4.5) {$\cdot$};
  \node at (1.5, 4.5) {$\cdot$};
  \node at (2.5, 4.5) {$\bullet$};
  \draw (0,5.5)--(0,0)--(4.5,0);
  \draw (0,1)--(3,1)--(3,2)--(0,2);
  \draw (1,1)--(1,5);
  \draw (2,1)--(2,2);
  \draw (0,3)--(1,3);
  \draw (0,4)--(2,4)--(2,3)--(1,3);
  \draw (0,5)--(3,5)--(3,4)--(2,4);
  \draw (2,4)--(2,5);
  \draw (2,3)--(4,3)--(4,2)--(3,2);
  \draw (2,2)--(2,3);
  \draw (3,2)--(3,3);
  \node at (2, -1) {\large $\mathsf{snow}(\mathbb{D}(\alpha))$};
\end{tikzpicture}}$$
        \caption{Snow diagram}
        \label{fig:krsnow}
    \end{figure}
\end{example}

\begin{remark}
    Theorem 1.1 of \textup{\textbf{\cite{Pan2}}} actually establishes a stronger result than that described in Theorem~\ref{thm:JP}. Among other things, Theorem 1.1 of \textup{\textbf{\cite{Pan2}}} establishes that not only can one determine $\mathrm{MaxG}(\mathbb{D}(\alpha))$ from the snow diagram of $\mathbb{D}(\alpha)$, but also the leading monomial in tail lexicographic order of the associated Lascoux polynomial $\mathfrak{L}_\alpha$.
\end{remark}

The following example shows that, unfortunately, Theorem~\ref{thm:JP} does not apply to all diagrams.

\begin{example}\label{ex:snowcounter}
    Let $D=\{(3,1),(2,2)\}$. The diagram $D$ along with its snow diagram $\mathsf{snow}(D)$ are illustrated in Figure~\ref{fig:snowcounter} below. While $\textup{sf}(D)=3$, one can compute that $\mathrm{MaxG}(D)=2$.
    \begin{figure}[H]
        \centering
        $$\scalebox{0.9}{\begin{tikzpicture}[scale=0.6]
        \node at (0.5, 2.5) {$\cdot$};
        \node at (1.5, 1.5) {$\cdot$};
  \draw (0,3.5)--(0,0)--(2.5,0);
  \draw (0,2)--(1,2)--(1,3)--(0,3);
  \draw (1,1)--(1,2)--(2,2)--(2,1)--(1,1);
  \node at (1, -1) {\large $D$};
\end{tikzpicture}}\quad\quad\quad\quad\quad\quad\quad\quad \scalebox{0.9}{\begin{tikzpicture}[scale=0.6]
        \node at (0.5, 2.5) {$\bullet$};
        \node at (1.5, 1.5) {$\bullet$};
  \draw (0,3.5)--(0,0)--(2.5,0);
  \draw (0,2)--(1,2)--(1,3)--(0,3);
  \draw (1,1)--(1,2)--(2,2)--(2,1)--(1,1);
        \node at (0.5, 0.5) {\scalebox{1.5}{$\ast$}};
        \node at (0.5, 1.5) {\scalebox{1.5}{$\ast$}};
        \node at (1.5, 0.5) {\scalebox{1.5}{$\ast$}};
  \node at (1, -1) {\large $SD(D)$};
\end{tikzpicture}}$$
        \caption{$\mathrm{MaxG}(D)\neq \text{sf}(D)$}
        \label{fig:snowcounter}
    \end{figure}
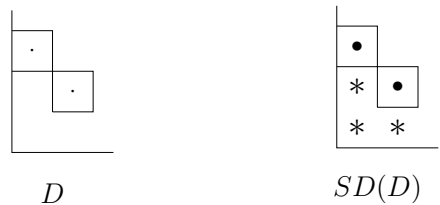
\end{example}

However, the authors of the present article believe that, in general, the value $\text{sf}(D)$ defined in Theorem~\ref{thm:JP} provides an upper bound on $\mathrm{MaxG}(D)$. In particular, we make the following conjecture.

\begin{conj}\label{conj:snow}
    If $D$ is a diagram that contains no ghost cells, then $\mathrm{MaxG}(D)\le \textup{sf}(D)$.
\end{conj}

In the section that follows, we establish means of computing $\mathrm{MaxG}(D)$ for additional special families of diagrams and, as a consequence, provide further evidence for Conjecture~\ref{conj:snow}.

\section{Extensions of Theorem~\ref{thm:JP}}\label{sec:extJP}

In this section, we extend Theorem~\ref{thm:JP} by identifying additional families of diagrams $D$ for which either $\mathrm{MaxG}(D)=\text{sf}(D)$ or a slight variation of $\mathsf{snow}(D)$ can be used to compute $\mathrm{MaxG}(D)$.  The main families of interest are the ``skew" and ``lock diagrams" of \textbf{\cite{KP1}}. In the case of skew diagrams, we show that $\mathrm{MaxG}(D)=\text{sf}(D)$. On the other hand, focusing on a subfamily of lock diagrams, we show that a slight variation of $\mathsf{snow}(D)$ can be used to compute $\mathrm{MaxG}(D)$; of note in this case is the application of a promising approach for establishing such results which makes use of labelings of cells. We start with skew diagrams. 

\begin{definition}
    For a weak composition $\alpha=(\alpha_1,\hdots,\alpha_n)$, the \textbf{skew diagram} $\mathbb{S}(\alpha)$ is constructed as follows:
\begin{itemize}
    \item left justify $\alpha_i$ cells in row $i$ for $i\in [n]$,
    \item for $j$ from $1$ to $n$ such that $\alpha_j>0$, take $i<j$ maximal such that $\alpha_i>0$, and if $\alpha_i>\alpha_j$, then shift rows $k\ge j$ rightward by $\alpha_i-\alpha_j$ columns,
    \item shift each row $j\in [n]$ rightward by $|\{i~|~1\le i<j,~\alpha_i=0\}|$ columns.
\end{itemize}
\end{definition}

\begin{example}
    Let $\alpha=(0,2,0,1,3)$. The skew diagram $\mathbb{S}(\alpha)$ is illustrated in Figure~\ref{fig:skew} below.
    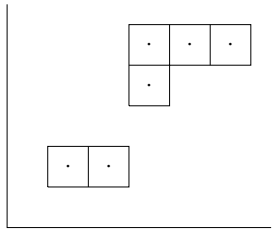
\begin{figure}[H]
        \centering
        $$\scalebox{0.9}{\begin{tikzpicture}[scale=0.6]
            \node at (1.5,1.5) {$\cdot$};
            \node at (2.5,1.5) {$\cdot$};
            \node at (3.5,3.5) {$\cdot$};
            \node at (3.5,4.5) {$\cdot$};
            \node at (4.5,4.5) {$\cdot$};
            \node at (5.5,4.5) {$\cdot$};
            \draw (6.5,0)--(0,0)--(0,5.5);
            \draw (1,1)--(3,1)--(3,2)--(1,2)--(1,1);
            \draw (2,1)--(2,2);
            \draw (3,3)--(4,3)--(4,4)--(6,4)--(6,5)--(3,5)--(3,3);
            \draw (3,4)--(4,4)--(4,5);
            \draw (5,4)--(5,5);
        \end{tikzpicture}}$$
        \caption{Skew diagram}
        \label{fig:skew}
    \end{figure}
\end{example}

\begin{remark}
    Skew diagrams were introduced in \textup{\textbf{\cite{KP1}}}, where the authors define skew polynomials to be the Kohnert poylnomials \textup(see Remark~\ref{rem:Kohnert}\textup) associated with skew diagrams. Moreover, they show that skew polynomials form a basis of $\mathbb{Z}[x_1,x_2,\hdots]$, have nonnegative expansions into demazure characters, and define polynomial generalizations of skew Schur functions in a precise sense. See \textup{\textbf{\cite{KP1}}} for complete details.
\end{remark}

The first main result of this section is as follows.

\begin{theorem}\label{thm:skew}
    If $D$ is a skew diagram, then $\mathrm{MaxG}(D)=\textup{sf}(D)$.
\end{theorem}

To prove Theorem~\ref{thm:skew}, we show that $\mathrm{MaxG}(D)=\text{sf}(D)$ holds for a more general family of diagrams which we are aptly calling ``generalized skew diagrams" and define as follows.

\begin{definition}\label{def:skew}
    Let $D$ be a diagram that contains no ghost cells. Suppose that the nonempty rows of $D$ are $\{r_i\}_{i=1}^n$ where $r_i<r_{i+1}$ for $i\in [n-1]$ when $n>1$, and that $(r_i,c_i^+)$ \textup(resp., $(r_i,c_i^-)$\textup) is the rightmost \textup(resp., leftmost\textup) cell in row $r_i$ of $D$ for $i\in [n]$. Then we refer to $D$ as a \textbf{generalized skew diagram} provided that for each $i\in [n]$ \begin{itemize}
        \item  $(r_i,\tilde{c})\in D$ for $c_i^-<\tilde{c}<c_i^+$ and
        \item $c_i^+\le c_j^+$ and $c_i^-\le c_j^-$ for $i<j\in [n]$.
    \end{itemize}  
\end{definition}

First, let us show that generalized skew diagrams do, in fact, generalize skew diagrams.

\begin{prop}\label{prop:skew}
    If $D$ is a skew diagram, then $D$ is a generalized skew diagram. Moreover, the reverse implication does not hold in general. 
\end{prop}
\begin{proof}
    If $D$ contains exactly one nonempty row, then the result is clear. So, assume that the nonempty rows of $D$ are $r_1<\cdots<r_n$ with $n>1$. Moreover, assume that $$a_i=|\{(r_i,c)\in D~|~c\ge 1\}|\quad\text{and}\quad N_i=\{(r_i,c)\notin D~|~c\ge 1~\text{and }\exists~c^*>c~\text{such that}~(r,c^*)\in D\}|$$ for $i\in [n]$, i.e., $a_i$ is the number of cells in row $r_i$ of $D$ and $N_i$ is the number of empty positions occurring to the left of those cells. Note that, using the notation of Definition~\ref{def:skew}, we have that $c_i^+=a_i+N_i$ and $c_i^-=N_i+1$ for $1\le i\le n$. Evidently, to show that $D$ is a generalized skew diagram, it suffices to show that $c_i^-\le c_{i+1}^-$ and $c_i^+\le c_{i+1}^+$ for $i\in [n-1]$. Considering the definition of skew diagram, it follows that $N_i\le N_{i+1}$ with $N_i+(a_i-a_{i+1})\le N_{i+1}$ if $a_i>a_{i+1}$. Consequently, for $i\in [n-1]$, we have $c_i^-=N_i+1\le N_{i+1}+1=c_{i+1}^-$. Moreover, if $a_i\le a_{i+1}$, then $$c_i^+=a_i+N_i\le a_{i+1}+N_{i+1}=c_{i+1}^+;$$ while if $a_i>a_{i+1}$, then $$c_i^+=a_i+N_i= a_{i+1}+N_{i}+(a_i-a_{i+1})\le a_{i+1}+N_{i+1}=c_{i+1}^+.$$ Thus, $D$ is a generalized skew diagram.

    To see that not all generalized skew diagrams are skew diagrams, consider the diagram $$\widehat{D}=\{(2,1),(2,2),(2,3),(2,4),(3,3),(3,4)\}$$ illustrated in Figure~\ref{fig:nskew}. 

    \begin{figure}[H]
        \centering
        $$\scalebox{0.9}{\begin{tikzpicture}[scale=0.6]
            \draw (5.5,0)--(0,0)--(0,3.5);
            \node at (0.5, 1.5) {$\cdot$};
            \node at (1.5, 1.5) {$\cdot$};
            \node at (2.5, 1.5) {$\cdot$};
            \node at (3.5, 1.5) {$\cdot$};
            \node at (2.5, 2.5) {$\cdot$};
            \node at (3.5, 2.5) {$\cdot$};
            \draw (0,1)--(4,1)--(4,3)--(2,3)--(2,2)--(0,2);
            \draw (1,1)--(1,2);
            \draw (2,1)--(2,2)--(4,2);
            \draw (3,1)--(3,3);
        \end{tikzpicture}}$$
        \caption{Generalized skew diagram}
        \label{fig:nskew}
    \end{figure}
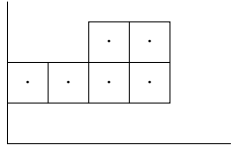

    \noindent
    The diagram $\widehat{D}$ is clearly a generalized skew diagram. On the other hand, since $\widehat{D}$ is not equal to $\mathbb{S}(0,4,2)=\{(2,2),(2,3),(2,4),(2,5),(3,6),(3,7)\}$, it follows that $\widehat{D}$ is not a skew diagram.
\end{proof}

As noted above, instead of proving Theorem~\ref{thm:skew} directly, we establish the following more general result. 


\begin{theorem}\label{thm:shift}
    If $D$ is a generalized skew diagram, then $\mathrm{MaxG}(D)= \textup{sf}(D)$.
\end{theorem}

Now, as a first step toward proving Theorem~\ref{thm:shift}, we first show that $\mathrm{MaxG}(D)\ge \text{sf}(D)$. Ongoing, we require the following notation. Given a diagram $D$ with nonempty columns $C=\{c_1<c_2<\hdots<c_n\}$ and $\widehat{C}\subseteq C$, define $$\text{flat}(D)=\bigcup_{i=1}^n\{(r,i)~|~(r,c_i)\in D\}\quad\quad\text{and}\quad\quad \text{Res}(D,\widehat{C})=\bigcup_{i=1}^n\{(r,c)\in D~|~c\in \widehat{C}\};$$ that is, $\text{flat}(D)$ is the diagram formed by left justifying the nonempty columns of $D$, while $Res(D,\widehat{C})$ is the restriction of $D$ to the columns in $\widehat{C}$.



\begin{prop}\label{prop:shiftmaxgesn}
 If $D$ is a generalized skew diagram, then $\mathrm{MaxG}(D)\ge \textup{sf}(D)$.
\end{prop}
\begin{proof}
    By induction on the number of nonempty columns of $D$. If $D$ has a single nonempty column, then $\text{flat}(D)$ is a key diagram. Considering the definitions of $K$-Kohnert moves and $\mathsf{snow}(D)$, we have that $\mathrm{MaxG}(D)=\mathrm{MaxG}(\text{flat}(D))$ and $\text{sf}(D)=\text{sf}(\text{flat}(D))$. Consequently, applying Theorem~\ref{thm:JP}, the result follows in the case where $D$ has a single nonempty column. Now, assume the result holds for generalized skew diagrams which contain up to $n-1\ge 1$ nonempty columns. Let $D$ be a generalized skew diagram with $n$ nonempty columns. It suffices to show that there exists a sequence of $K$-Kohnert moves that one can apply to $D$ to form $D^*\in KKD(D)$ satisfying $|G(D^*)|=\text{sf}(D)$. To this end, we first define an intermediate diagram $\widehat{D}\in KKD(D)$. Suppose that $c^*$ is minimal such that there exists $r^*$ satisfying $(r^*,c^*)\in \mathsf{dark}(D)$. Note that, since $D$ is a generalized skew diagram, $|Dk(D)|>1$ and $(r,c)\in \mathsf{dark}(D)$ for $c>c^*$ implies $r>r^*$. Assume that there exists $r<r^*$ such that $(r,c^*)\notin D$; the other case following via a similar but simpler argument taking $\widehat{D}=D$. Let $\widehat{r}<r^*$ be maximal such that $(\widehat{r},c^*)\notin D$. Note that, considering the definition of generalized skew diagram, we have 
    \begin{equation}\label{eq:propmxgesn}
        |\{r~|~r\le \widehat{r}~\text{and}~(r,c^*)\notin D\}|=|\{r~|~r< r^*~\text{and}~(r,c^*)\notin D\}|=\widehat{r}
    \end{equation}
    and $(r,c)\notin D$ for $r\le\widehat{r}$ and $c\ge c^*$. Setting $N=|\{c~|~c\ge c^*~\text{and}~(\widehat{r}+1,c)\in D\}|$, where $N\ge 1$ by assumption, define $\widehat{D}$ to be the diagram formed from $D$ as follows. For $r$ satisfying $2\le r\le \widehat{r}+1$ in decreasing order, apply in succession $N-1$ Kohnert moves followed by a single ghost move at row $r$. By construction, $\langle\tilde{r},c\rangle\in\widehat{D}$ for $2\le \tilde{r}\le \widehat{r}+1$; that is, $D$ has $\widehat{r}$ ghost cells in column $c^*$, which, considering (\ref{eq:propmxgesn}), is equal to the number of empty positions below the cell in position $(r^*,c^*)$ of $D$. Thus, in $\widehat{D}$ we have a diagram formed from $D$ which contains the same number of ghost cells in column $c^*$ as snowflakes in column $c^*$ of $\mathsf{snow}(D)$. Considering our choice of $c^*$, it remains to show that there exists a sequence of $K$-Kohnert moves which, when applied to $\widehat{D}$, add as many ghost cells as there are snowflakes in columns $c>c^*$ of $\mathsf{snow}(D)$.
    
    Taking $\widehat{c}$ to be maximal such that column $\widehat{c}$ of $D$ is nonempty and letting $$C=\{c~|~c^*< c\le \widehat{c}\},\quad T=Res(D,C),\quad\text{and}\quad \widehat{T}=Res(\widehat{D},C),$$ we have that $T$ is a generalized skew diagram and $\widehat{T}$ is equal to either $T$ or $T$ with its lowest nonempty row bottom justified. Consequently, $\widehat{T}$ is also a skew diagram. Moreover, since $$(\tilde{r},\tilde{c})\in S=\{(r,c)\in \mathsf{dark}(D)~|~c>c^*\}$$ implies $r>r^*$, it follows that $S=\mathsf{dark}(T)=\mathsf{dark}(\widehat{T})$. Note that in moving from $D$ to $T$ and $\widehat{T}$, the number of empty positions below cells of $S$ was left unchanged. Consequently, applying our induction hypothesis, there exists a sequence of $K$-Kohnert moves one can apply to $\widehat{T}$ to form a diagram $T^*$ such that $$|G(T^*)|=\text{sf}(T^*)=\text{sf}(T).$$ Evidently, applying this same sequence of $K$-Kohnert moves to $\widehat{D}$ results in the desired diagram $D^*$. The result follows. 
\end{proof}

To prove Theorem~\ref{thm:shift}, it remains to show that $\mathrm{MaxG}(D)\le \text{sf}(D)$. As an intermediate step, in Lemma~\ref{lem:shift} below we show that Theorem~\ref{thm:shift} holds for a restricted class of generalized skew diagrams. For the sake of defining the restricted class of interest, given a generalized skew diagram $D$, let $Key(D)$ denote the minimal key diagram containing the cells of $D$. Formally, $$Key(D)=\{(r,c)~|~\exists~\widehat{c}\ge c~\text{such that}~(r,\widehat{c})\in D\}.$$

\begin{lemma}\label{lem:shift}
    Let $D$ be a generalized skew diagram. If $\mathsf{dark}(Key(D))=\mathsf{dark}(D)$, then $\mathrm{MaxG}(D)=\textup{sf}(D)$. 
\end{lemma}
\begin{proof}
    Assume that the nonempty rows of $D$ are $\{r_i\}_{i=1}^n$, where $r_i<r_{i+1}$ for $i\in [n-1]$ in the case $n>1$. Considering Proposition~\ref{prop:shiftmaxgesn}, it suffices to show that $\mathrm{MaxG}(D)\le \text{sf}(D)$. Recall that $c_i^-=\min\{c~|~(r_i,c)\in D\}$ for $i\in [n]$. To start, we claim that $\text{sf}(D)=\text{sf}(Key(D))$. Assume otherwise. Then since $\mathsf{dark}(Key(D))=\mathsf{dark}(D)$ and $D\subseteq Key(D)$, it must be the case that $\text{sf}(Key(D))<\text{sf}(D)$. Considering the definition of $\text{sf}(-)$, this implies that there exists $i,j\in [n]$ and $c$ such that $i>j$, $(r_j,c)\in Key(D)$, $(r_j,c)\notin D$, and $(r_i,c)\in D\cap Key(D)$. Note that, considering the definition of $Key(D)$, $(r_j,c)\in Key(D)$ and $(r_j,c)\notin D$ together imply that there exists $\tilde{c}>c$ for which $(r_j,\tilde{c})\in D$; but then $c_j^-=\min\{\tilde{c}>c~|~(r_j,\tilde{c})\in D\}>c\ge c_i^-$, which is a contradiction. Thus, the claim follows. 

    Now, consider $S=Key(D)\backslash D$. Note that, by the definition of $Key(D)$ and the fact that $c_i^-\le c_{i+1}^-$ for $i\in [n-1]$ when $n>1$, if $(r,c)\in D$ and $(\tilde{r},\tilde{c})\in S$, then $\tilde{r}\ge r$ and $\tilde{c}\le c$ with at least one inequality strict; that is, $(\ast)$ no cell of $S$ lies weakly southeast of a cell in $D$. We claim that if $T\in KKD(D)$, then $T\cup S\in KKD(Key(D))$. To establish the claim, since $D\cup S=Key(D)\in KKD(Key(D))$, it suffices to show that if $T\in KKD(D)$ satisfies $T\cup S\in KKD(Key(D))$ and $r$ is such that $T\neq \mathcal{K}(T,r)$, then $\mathcal{K}(T,r)\cup S=\mathcal{K}(T\cup S,r)$ and $\mathcal{G}(T,r)\cup S=\mathcal{G}(T\cup S,r)$. Take such a $T\in KKD(D)$ and $r$. Assume that $\mathcal{K}(T,r)\cup S\neq \mathcal{K}(T\cup S,r)$ with $\mathcal{K}(T,r)=T\ldownarrow^{(r,c)}_{(\widehat{r},c)}$. Then either
    \begin{enumerate}
        \item[\textup{(1)}] $(r,c)$ is not rightmost in row $r$ of $T\cup S$ or 
        \item[\textup{(2)}] $(\widehat{r},c)\in T\cup S$.
    \end{enumerate}
    Note that both (1) and (2) imply that there is a cell of $S$ lying weakly southeast of a cell of $D$, contradicting $(\ast)$. A similar argument applies when assuming there exists $T\in KKD(D)$ and $r$ for which $\mathcal{G}(T,r)\cup S\neq \mathcal{G}(T\cup S,r)$. Thus, our claim follows. Consequently, $\mathrm{MaxG}(D)\le \mathrm{MaxG}(Key(D))$ so that, applying Theorem~\ref{thm:JP}, $$\mathrm{MaxG}(D)\le \mathrm{MaxG}(Key(D))=\text{sf}(Key(D))=\text{sf}(D).$$ The result follows.
\end{proof}

To finish the proof Theorem~\ref{thm:shift}, it remains to consider the case where $\mathsf{dark}(Key(D))\neq \mathsf{dark}(D)$. The following proposition allows us to do so.

\begin{prop}\label{prop:JPCons}
    Let $D$ be a diagram that contains no ghost cells. Assume that there exists a partition $\{C_i\}_{i=1}^n$ of the nonempty columns of $D$ such that, for $i\in [n]$,
    \begin{enumerate}
        \item[\textup{(a)}] $F_i=\text{flat}(\text{Res}(D,C_i))$ is either a key diagram or a skew diagram with $\mathsf{dark}(Key(F_i))=\mathsf{dark}(F_i)$;
        \item[\textup{(b)}] for $c\in C_i$, there exists $r$ such that $(r,c)\in \mathsf{dark}(\text{Res}(D,C_i))$ if and only if there exists $\widehat{r}$ such that $(\widehat{r},c)\in \mathsf{dark}(D)$; and
        \item[\textup{(c)}] for $c\in C_i$, if $(r,c)\in \mathsf{dark}(\text{Res}(D,C_i))$ and $(\widehat{r},c)\in \mathsf{dark}(D)$, then $\widehat{r}\le r$ and $(\tilde{r},c)\in D$ for $\widehat{r}\le \tilde{r}\le r$.
    \end{enumerate}
    Then $\mathrm{MaxG}(D)\le \textup{sf}(D)$.  
\end{prop}
\begin{proof}
    Let $D_i=\text{Res}(D,C_i)$ for $i\in [n]$. Note that, considering the definition of $K$-Kohnert move, if a diagram $T$ is related to a diagram $T^*$ by the addition or removal of empty columns, then $\mathrm{MaxG}(T)=\mathrm{MaxG}(T^*)$. Consequently, considering property (a) of $C_i$ and applying either Theorem~\ref{thm:JP} or Lemma~\ref{lem:shift}, we have that $\mathrm{MaxG}(D_i)=\text{sf}(D_i)$ for $i\in [n]$. Now, evidently, $$\mathrm{MaxG}(D)\le \sum_{i=1}^n\mathrm{MaxG}(D_i)= \sum_{i=1}^n\text{sf}(D_i).$$ Thus, to finish the proof, it remains to show that $\sum_{i=1}^n\text{sf}(D_i)=\text{sf}(D)$. Define $$\mathsf{dark}^+=\bigcup_{i=1}^n\{(r,c,i)~|~(r,c)\in \mathsf{dark}(D_i)\}.$$ By property (b), there exists a bijection $\phi:\mathsf{dark}(D)\to \mathsf{dark}^+$. Moreover, as a consequence of a combination of properties (b) and (c), $\phi$ can be defined in such a way that for $(r,c)\in \mathsf{dark}(D)$ we have $\phi((r,c))=(r',c,i)\in \mathsf{dark}(D_i)$ for some $i\in [n]$ and the number of empty positions below $(r,c)$ in $D$ is equal to the number below $(r',c)$ in $D_i$. Since the number of snowflakes in a snow diagram is equal to the total number of empty positions contained below dark clouds, it follows that $\sum_{i=1}^n\text{sf}(D_i)=\text{sf}(D)$.
\end{proof}

\begin{corollary}\label{cor:UB}
    Suppose that $D$ is a diagram that contains no ghost cells for which either
    \begin{itemize}
        \item[\textup{(a)}] each column contains at most one cell or
        \item[\textup{(b)}] there exists $n\in\mathbb{Z}_{>0}$ such that $(r,c)\in D$ if and only if $r,c\in [n]$ and $r+c$ is even \textup(odd\textup).
    \end{itemize}
    Then $\mathrm{MaxG}(D)\le \text{sf}(D)$.
\end{corollary}
\begin{proof}
    For (a), if $\{r_i\}_{i=1}^n=\{r~|~\exists~c~\text{such that}~(r,c)\in D\}$, then define $C_i=\{c~|~(r_i,c)\in D\}$ for $i\in [n]$. As for (b), define $C_1=\{c\in [n]~|~c~\text{odd}\}$ and $C_2=\{c\in [n]~|~c~\text{even}\}$.
\end{proof}

The inequality in the conclusion of Proposition~\ref{prop:JPCons} can be strict, as illustrated in Example~\ref{ex:strict}.

\begin{example}\label{ex:strict}
    Let $D_1=\{(1,3),(2,2),(3,1)\}$ and $D_2=\{(3,1),(3,2),(3,3),(4,1),(4,3),(5,3)\}$, illustrated below in Figures~\ref{fig:swcounter} \textup(a\textup) and \textup(b\textup), respectively. It is straightforward to verify that for $j\in [2]$, $D_j$ along with $C_i=\{i\}$ for $i\in [3]$ satisfy the hypotheses of Proposition~\ref{prop:JPCons} so that $\mathrm{MaxG}(D_j)\le \textup{sf}(D_j)$. For both $j=1$ and $2$, this inequality is, in fact, strict as one can compute that $\mathrm{MaxG}(D_1)=2<3=\textup{sf}(D_1)$ and $\mathrm{MaxG}(D_2)=5<6=\textup{sf}(D_2)$.
    \begin{figure}[H]
        \centering
        $$\scalebox{0.9}{\begin{tikzpicture}[scale=0.6]
        \node at (0.5, 2.5) {$\cdot$};
        \node at (1.5, 1.5) {$\cdot$};
        \node at (2.5, 0.5) {$\cdot$};
        \draw (0,3.5)--(0,0)--(3.5,0);
        \draw (0,2)--(1,2)--(1,3)--(0,3);
        \draw (1,1)--(2,1)--(2,2)--(1,2)--(1,1);
        \draw (2,0)--(2,1)--(3,1)--(3,0);
        \node at (1.5,-1) {\Large (a)};
\end{tikzpicture}}\quad\quad\quad\quad\quad\quad\quad\quad\scalebox{0.9}{\begin{tikzpicture}[scale=0.6]
        \node at (0.5, 2.5) {$\cdot$};
        \node at (0.5, 3.5) {$\cdot$};
        \node at (1.5, 2.5) {$\cdot$};
        \node at (2.5, 2.5) {$\cdot$};
        \node at (2.5, 3.5) {$\cdot$};
        \node at (2.5, 4.5) {$\cdot$};
        \draw (0,5.5)--(0,0)--(3.5,0);
        \draw (0,2)--(3,2)--(3,5)--(2,5)--(2,3)--(1,3)--(1,4)--(0,4);
        \draw (0,3)--(1,3)--(1,2);
        \draw (2,2)--(2,3)--(3,3);
        \draw (2,4)--(3,4);
        \node at (2,-1) {\Large (b)};
\end{tikzpicture}}$$
        \caption{Diagrams $D$ with $\mathrm{MaxG}(D)<\text{sf}(D)$}
        \label{fig:swcounter}
    \end{figure}
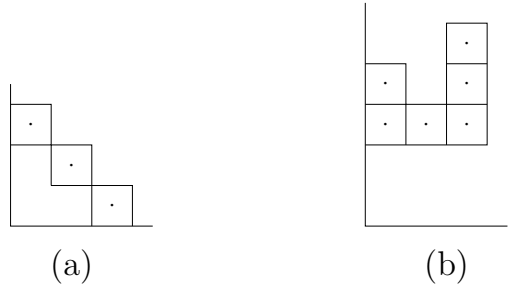
\end{example}

\begin{remark}
Note that the diagrams considered in Corollary~\ref{cor:UB} \textup{(b)} are the ``checkered diagrams" of \textup{\textbf{\cite{KPoset1}}}. The authors of the present article claim that for such diagrams $D$, one has $\mathrm{MaxG}(D)=\textup{sf}(D)$. Considering Corollary~\ref{cor:UB}, to prove that $\mathrm{MaxG}(D)=\textup{sf}(D)$ for checkered diagrams, it suffices to show that there exists $T\in KKD(D)$ with $|G(T)|=\textup{sf}(D)$; for the sake of brevity, the details for constructing such a $T\in KKD(D)$ are left to the interested reader.
\end{remark}

Proposition~\ref{prop:JPCons} in hand, we can now finish the proof of Theorem~\ref{thm:shift}.

\begin{proof}[Proof of Theorem~\ref{thm:shift}]
    Let $\{r_i\}_{i=1}^n$ be the nonempty rows of $D$ where $r_i<r_{i+1}$ for $i\in [n-1]$ when $n>1$. If $\mathsf{dark}(Key(D))=\mathsf{dark}(D)$, then the result follows by Lemma~\ref{lem:shift}. So, assume that $\mathsf{dark}(Key(D))\neq \mathsf{dark}(D)$; note that, as a consequence, $n>1$. We define a partitioning of the columns of $D$ as follows. Let $R\subset \{r_i\}_{i=1}^n$ consist of $r_n$ as well as all of those $r_j$ for $j\in [n-1]$ such that
    \begin{itemize}
        \item row $r_j$ of $\mathsf{snow}(D)$ contains a dark cloud and
        \item row $r_{j+1}$ of $\mathsf{snow}(D)$ contains no dark cloud.
    \end{itemize}
Assume that $R=\{r_{i_j}\}_{j=1}^m$ where $r_j<r_{j+1}$ for $j\in [m-1]$ when $m>1$, and define $c_{i_j}$ for $j\in [m]$ by $(r_{i_j},c_{i_j})\in \mathsf{dark}(D)$. Considering the definitions of $\mathsf{dark}(D)$ and generalized skew diagram, we have that $c_{i_j}<c_{i_{j+1}}$ for $j\in [m-1]$ when $m>1$. Letting $c_{i_0}=0$, define $$C_j=\{c~|~c_{i_{j-1}}<c\le c_{i_j}\}$$ for $j\in [m]$. We claim that $D$ along with $\{C_j\}_{j=1}^m$ satisfy the hypotheses of Proposition~\ref{prop:JPCons} so that $\mathrm{MaxG}(D)\le \text{sf}(D)$. Considering the definitions of generalized skew diagram and the $C_j$, it follows immediately that $D_j=Res(D,C_j)$ is a generalized skew diagram for $j\in [m]$. Thus, to establish the claim, it remains to show that $D$ along with $\{C_j\}_{j=1}^m$ satisfy hypotheses (b) and (c) of Proposition~\ref{prop:JPCons}. To this end, we show that $\mathsf{dark}(D)\cap\{(r,c)~|~c_{i_{j-1}}<c\le c_{i_j},~r\ge 1\}=\mathsf{dark}(Res(D,D_j))$ for $j\in [m]$, from which it follows that both remaining hypotheses are satisfied by $D$ along with $\{C_j\}_{j=1}^m$. Note that $(r_{i_j},c_{i_j})$ must be the top rightmost cell in $Res(D,C_j)$ for $j\in [m]$; this is immediate for $j=m$. If $m>1$ and $(r_{i_j},c_{i_j})$ were not the top rightmost cell in $Res(D,C_j)$ for $j\in [m-1]$, then $(r_{i_j+1},c_{i_j})\in D$ and, since row $r_{i_j+1}$ of $\mathsf{snow}(D)$ contains no dark cloud, there exists $\tilde{r}>r_{i_j+1}$ such that $(\tilde{r},c_{i_j})\in \mathsf{dark}(D)$, contradicting the definition of $c_{i_j}$. Thus, $(r_{i_j},c_{i_j})\in \mathsf{dark}(D)\cap Res(D,C_j)$ for $j\in [m]$. Moreover, it follows that $(r,c)\notin D$ for $r>r_{i_j}$ and $c\le c_{i_j}$. Consequently, $$\mathsf{dark}(D)\cap\{(r,c)~|~c_{i_{j-1}}<c\le c_{i_j}\},~\mathsf{dark}(Res(D,D_j))\subset\{(r,c)~|~r\le r_{i_j},~c\le c_{i_j}\}$$ for $j\in [m]$. Therefore, for $j\in [m]$, $D$ and $\mathsf{dark}(Res(D,D_j))$ have the same restrictions defining positions of dark clouds in columns $c$ satisfying $c_{i_{j-1}}<c\le c_{i_j}$, i.e., $$\mathsf{dark}(D)\cap\{(r,c)~|~c_{i_{j-1}}<c\le c_{i_j}\}=\mathsf{dark}(Res(D,D_j)),$$ establishing the claim. As noted above, we may conclude that $\mathrm{MaxG}(D)\le \text{sf}(D)$. Applying Proposition~\ref{prop:shiftmaxgesn}, the result follows.
\end{proof}

As noted above, Theorem~\ref{thm:skew} follows as an immediate corollary of Theorem~\ref{thm:shift}.

\begin{remark}
    For the diagrams $D$ considered in Theorem~\ref{thm:shift}, there exist weak compositions $\alpha$ and $\beta$ such that $Key(D)=\mathbb{D}(\alpha)$ and $Key(D)\backslash D=\mathbb{D}(\beta)$. In particular, using the notation of Definition~\ref{def:skew}, $$\alpha=(\underbrace{0,\hdots,0}_{r_1-1},\underbrace{c_1^+,\hdots,c_1^+}_{r_2-r_1},\hdots,\underbrace{c_n^+,\hdots,c_n^+}_{r_n-r_{n-1}})\quad\text{and}\quad \beta=(\underbrace{0,\hdots,0}_{r_1-1},\underbrace{c_1^-,\hdots,c_1^-}_{r_2-r_1},\hdots,\underbrace{c_n^-,\hdots,c_n^-}_{r_n-r_{n-1}}).$$ Note that the entries of both $\alpha$ and $\beta$ are weakly increasing from left to right. Thus, it is natural to ask whether Theorem~\ref{thm:shift} holds if one replaces weakly increasing $\alpha$ and $\beta$ by arbitrary weak compositions. As it turns out, as we will see shortly, such a generalization of Theorem~\ref{thm:shift} does not hold in general.
\end{remark}

To end this section, we demonstrate what seems to be a promising approach for establishing further results in the spirit of Theorem~\ref{thm:JP} using labelings. In particular, we apply such an approach to a special class of ``lock diagrams" using an extended version of a labeling defined in \textbf{\cite{KP1}}; the resulting formula for $\mathrm{MaxG}(D)$ requires a variation of $\mathsf{snow}(D)$. Moreover, we indicate how one could extend this approach to more general diagrams.




\begin{definition}
Given a weak composition $\alpha=(\alpha_1,\hdots,\alpha_n)\in\mathbb{Z}_{\ge 0}^n$ and letting $$N=\max(\alpha)=\max\{\alpha_i~|~1\le i\le n\}$$ we define the \textbf{lock diagram} associated with $\alpha$ as $$\rotatebox[origin=c]{180}{$\mathbb{D}$}(\alpha)=\bigcup_{i=1}^n\{(i,N-j)~|~0\le j\le \alpha_i-1\}.$$ Moreover, we define a \textbf{lock tableau} of content $\alpha$ to be a diagram filled with entries $1^{\alpha_1},2^{\alpha_2},\hdots,n^{\alpha_n}$, one per cell, satisfying the following conditions:
\begin{enumerate}
    \item[\textup(i\textup)] if $\alpha_j>0$ for $j\in [n]$, then there is exactly one $j$ in each column from $N-\alpha_j+1$ through $\max(\alpha)$;
    \item[\textup(ii\textup)] each entry in row $r$ is at least $r$;
    \item[\textup(iii\textup)] the cells with entry $j$ weakly descend from left to right; and
    \item[\textup(iv\textup)] the labeling strictly decreases down columns.
\end{enumerate}
\end{definition}

In \textbf{\cite{KP1}}, the authors establish the following relationship between the elements of $KD(\rotatebox[origin=c]{180}{$\mathbb{D}$}(\alpha))$ and lock tableau of content $\alpha$.

\begin{theorem}[Theorem 6.9, \textbf{\cite{KP1}}]\label{thm:label}
    The underling diagrams of lock tableau with content $\alpha$ are exactly the elements of $KD(\rotatebox[origin=c]{180}{$\mathbb{D}$}(\alpha))$.
\end{theorem}

As a consequence of Theorem~\ref{thm:label}, for a lock diagram $D$, we can think of the cells of $T\in KD(D)$ as having a natural labeling, i.e., that of the associated lock tableau. Since $T\in KKD(D)$ implies that $T\cap\{(r,c)~|~r>0,~c>0\}\in KD(D)$, the non-ghost cells of $KKD(D)$ inherit the labeling obtained via Theorem~\ref{thm:label}. Ongoing, we denote the label of a cell $(r,c)\in T$ for $T\in KKD(D)$ as $L_T(r,c)$; our notation does not make $D$ explicit as this will be clear from context. We extend the labeling obtained via Theorem~\ref{thm:label} to all cells of $T\in KKD(D)$ as follows. If $\langle r,c\rangle\in T$ and $r^*\le r$ is maximal such that $(r^*,c)\in T$, then $L_T(r,c)=L_T(r^*,c)$.

\begin{example}
In Figure~\ref{fig:locklab} \textup{(b)} and \textup{(c)} below, we illustrate the labelings of the diagrams $T_1,T_2\in KKD(D)$ for $D=\rotatebox[origin=c]{180}{$\mathbb{D}$}(0,0,3,2,3)$ the lock diagram of Figure~\ref{fig:locklab} \textup{(a)}. The labels of the ghost cells in Figure~\ref{fig:locklab} \textup{(c)} are placed in the top right corners of the cells.
\begin{figure}[H]
    \centering
    $$\scalebox{0.8}{\begin{tikzpicture}[scale=0.65]
        \node at (0.5,2.5) {$\cdot$};
        \node at (0.5,4.5) {$\cdot$};
        \node at (1.5,2.5) {$\cdot$};
        \node at (1.5,3.5) {$\cdot$};
        \node at (1.5,4.5) {$\cdot$};
        \node at (2.5,2.5) {$\cdot$};
        \node at (2.5,3.5) {$\cdot$};
        \node at (2.5,4.5) {$\cdot$};
        \draw (0,5.5)--(0,0)--(3.5,0);
        \draw (0,2)--(3,2)--(3,3)--(0,3);
        \draw (1,2)--(1,3);
        \draw (2,2)--(2,3);
        \draw (1,3)--(1,4);
        \draw (2,3)--(2,4);
        \draw (3,3)--(3,4);
        \draw (0,4)--(3,4)--(3,5)--(0,5);
        \draw (1,4)--(1,5);
        \draw (2,4)--(2,5);
        \node at (1.5,-1) {\Large (a)};
    \end{tikzpicture}}\quad\quad\quad \scalebox{0.8}{\begin{tikzpicture}[scale=0.65]
        \node at (1.5,0.5) {$3$};
        \node at (2.5,0.5) {$3$};
        \node at (2.5,1.5) {$4$};
        \node at (0.5, 2.5) {$3$};
        \node at (1.5, 2.5) {$4$};
        \node at (2.5, 2.5) {$5$};
        \node at (0.5, 4.5) {$5$};
        \node at (1.5, 4.5) {$5$};
        \draw (0,5.5)--(0,0)--(3.5,0);
        \draw (1,0)--(1,1)--(3,1)--(3,0);
        \draw (2,0)--(2,1);
        \draw (2,1)--(2,3)--(3,3)--(3,1);
        \draw (2,2)--(3,2);
        \draw (0,2)--(2,2);
        \draw (1,2)--(1,3);
        \draw (0,3)--(2,3);
        \draw (0,4)--(2,4)--(2,5)--(0,5);
        \draw (1,4)--(1,5);
        \node at (1.5,-1) {\Large (b)};
    \end{tikzpicture}}\quad\quad\quad \scalebox{0.8}{\begin{tikzpicture}[scale=0.65]
        \node at (1.5,0.5) {$3$};
        \node at (2.5,0.5) {$3$};
        \node at (2.5,1.5) {$4$};
        \node at (0.5, 2.5) {$3$};
        \node at (1.5, 2.5) {$4$};
        \node at (2.5, 2.5) {$5$};
        \node at (0.5, 4.5) {$5$};
        \node at (1.5, 4.5) {$5$};
        \draw (0,5.5)--(0,0)--(3.5,0);
        \draw (1,0)--(1,1)--(3,1)--(3,0);
        \draw (2,0)--(2,1);
        \draw (2,1)--(2,3)--(3,3)--(3,1);
        \draw (2,2)--(3,2);
        \draw (0,2)--(2,2);
        \draw (1,2)--(1,3);
        \draw (0,3)--(2,3);
        \draw (0,4)--(2,4)--(2,5)--(0,5);
        \draw (1,4)--(1,5);
        \draw (1,1)--(1,2);
        \node at (1.5, 1.5) {$\bigtimes^3$};
        \draw (1,3)--(1,4);
        \draw (2,3)--(2,4);
        \node at (1.5, 3.5) {$\bigtimes^4$};
        \draw (2,3)--(2,4)--(3,4)--(3,3);
        \node at (2.5, 3.5) {$\bigtimes^5$};
        \node at (1.5,-1) {\Large (c)};
    \end{tikzpicture}}$$
    \caption{Labelings}
    \label{fig:locklab}
\end{figure}
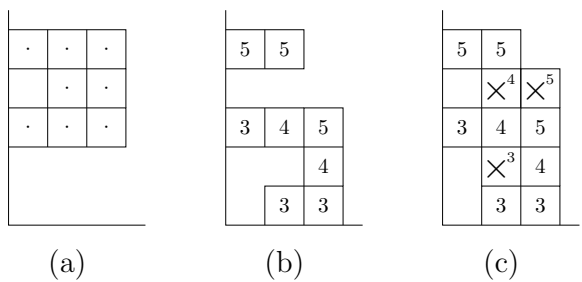
\end{example}

Given a diagram $D$, define $\mathsf{snow}^*(D)$ to be the diagram formed form $\mathsf{snow}(D)$ by removing any snowflakes in nonempty rows of $D$. For lock diagrams $D$, using our labeling of diagrams in $KKD(D)$, we establish the following.

\begin{theorem}\label{thm:lockmain}
    Let $D$ be a lock diagram with nonempty rows $\{r_i\}_{i=1}^n$ where $r_i<r_{i+1}$ for $i\in [n-1]$ when $n>1$. Suppose that there exists at most one $r_k\in\{r_i~|~r_i>i\}$ such that row $r_k$ of $\mathsf{snow}(D)$ contains no dark clouds. Then $\mathrm{MaxG}(D)$ is equal to the number of snowflakes in $\mathsf{snow}^*(D)$.
\end{theorem}

\begin{remark}
    There exist diagrams considered in Theorem~\ref{thm:lockmain} for which $\mathrm{MaxG}(D)<\textup{sf}(D)$. For example, this is the case for $D=\rotatebox[origin=c]{180}{$\mathbb{D}$}(\alpha)$ in Example~\ref{ex:whsd} below.
\end{remark}

\begin{example}\label{ex:whsd}
For $\alpha=(0,4,0,2,3,2,1)$ and $D=\rotatebox[origin=c]{180}{$\mathbb{D}$}(\alpha)$, the diagram $\mathsf{snow}^*(D)$ is illustrated in Figure~\ref{fig:whsd}. Applying Theorem~\ref{thm:lockmain} we find that $\mathrm{MaxG}(D)=7<8=\textup{sf}(D)$.
\begin{figure}[H]
    \centering
    $$\scalebox{0.8}{\begin{tikzpicture}[scale=0.65]
        \node at (3.5,6.5) {$\bullet$};
        \node at (2.5,5.5) {$\bullet$};
        \node at (1.5,4.5) {$\bullet$};
        \node at (0.5,1.5) {$\bullet$};
        \node at (1.5,2.5) {\scalebox{1.5}{$\ast$}};
        \node at (2.5,2.5) {\scalebox{1.5}{$\ast$}};
        \node at (3.5,2.5) {\scalebox{1.5}{$\ast$}};
        \node at (0.5,0.5) {\scalebox{1.5}{$\ast$}};
        \node at (1.5,0.5) {\scalebox{1.5}{$\ast$}};
        \node at (2.5,0.5) {\scalebox{1.5}{$\ast$}};
        \node at (3.5,0.5) {\scalebox{1.5}{$\ast$}};
        \draw (0,7.5)--(0,0)--(4.5,0);
        \draw (0,1)--(4,1)--(4,2)--(0,2);
        \draw (1,1)--(1,2);
        \draw (2,1)--(2,2);
        \draw (3,1)--(3,2);
        \draw (2,3)--(4,3)--(4,4)--(2,4)--(2,3);
        \draw (3,3)--(3,4);
        \draw (4,4)--(4,5)--(1,5)--(1,4)--(2,4);
        \draw (2,4)--(2,5);
        \draw (3,4)--(3,5);
        \draw (4,5)--(4,6)--(2,6)--(2,5);
        \draw (3,5)--(3,6);
        \draw (4,6)--(4,7)--(3,7)--(3,6);
    \end{tikzpicture}}$$
    \caption{$\mathsf{snow}^*(D)$}
    \label{fig:whsd}
\end{figure}
\end{example}

For the proof of Theorem~\ref{thm:lockmain}, we require the following lemma which records a number of results concerning our labeling of cells in diagrams $T\in KKD(D)$ where $D$ is a lock diagram. As the corresponding notation proves to be useful, we note that one can define a poset structure on $KKD(D)$ where for $D_1,D_2\in KKD(D)$ we say $D_2\prec D_1$ if $D_2$ can be obtained from $D_1$ by applying some sequence of $K$-Kohnert moves. In the case that $D_1$ covers $D_2$ in this ordering, i.e., $D_2\prec D_1$ and there exists no $\tilde{D}\in KKD(D)$ such that $D_2\prec\tilde{D}\prec D_1$, then we write $D_2\precdot D_1$. Note that if $D_2\precdot D_1$, then $D_2$ can be formed from $D_1$ by applying a single $K$-Kohnert move.

\begin{lemma}\label{lem:colorp}
    Let $D$ be a lock diagram and $T\in KKD(D)$.
    \begin{enumerate}
        \item[\textup{(a)}] If column $c$ of $D$ is nonempty, then $\{r~|~(r,c)\in D\}=\{L_T(r,c)~|~(r,c)\in T\}$.

        \item[\textup{(b)}] If $T\in KKD(D)$ and $(r,c)\in T$, then $r\le L_T(r,c)$.

        \item[\textup{(c)}] If $(r_1,c),(r_2,c)\in T$ with $r_1<r_2$, then $L_T(r_1,c)<L_T(r_2,c)$.

        \item[\textup{(d)}] If $(r_1,c_1),(r_2,c_2)\in T$ and $L_T(r_1,c_1)=L_T(r_2,c_2)$, then $c_1\neq c_2$. Moreover, if $c_1<c_2$, then $r_1\ge r_2$.

        \item[\textup{(e)}] If $\langle r,c\rangle\in T$, then $\langle r,c\rangle\in \tilde{T}$ and $L_{\tilde{T}}(r,c)=L_T(r,c)$ for all $\tilde{T}\prec T$.

        \item[\textup{(f)}] If $\widehat{T}=\mathcal{G}(T,r)=T\ldownarrow^{(r,c)}_{(\widehat{r},c)}\cup\{\langle r,c\rangle\}$, then $L_{\widehat{T}}(r,c)=L_T(r,c)$.

        \item[\textup{(g)}] If $\langle r,c\rangle\in T$, then $r\le L_T(r,c)$.


        \item[\textup{(h)}] If $\langle r,c_1\rangle,\langle r,c_2\rangle\in T$ with $c_1<c_2$, then $L_T(r,c_1)<L_T(r,c_2)$.

        \item[\textup{(i)}] If $\langle r,c_1\rangle,\langle r+1,c_2\rangle\in T$ with $L_T(r,c_1)<L_T(r+1,c_2)$, then $c_1<c_2$.
    \end{enumerate}
\end{lemma}
\begin{proof}
(a), (b), (c), and (d) are properties (i), (ii), (iii), and (iv), respectively, of the labeling defined in \textbf{\cite{KP1}} for diagrams $T\in KD(D)\subseteq KKD(D)$ expressed using the notation $L_T(r,c)$; these properties evidently carry over to our extended labeling for diagrams $T\in KKD(D)$.    

(e) The fact that $\langle r,c\rangle\in \tilde{T}$ for all $\tilde{T}\prec T$ follows by the definition of $K$-Kohnert move. As for the label of $\langle r,c\rangle$, once again considering the definition of $K$-Kohnert move, it follows that $$|\{(\tilde{r},c)\in T~|~\tilde{r}<r\}|=|\{(\tilde{r},c)\in \tilde{T}~|~\tilde{r}<r\}|$$ for all $\tilde{T}\prec T$, i.e., the number of non-ghost cells lying below a ghost cell is preserved by $K$-Kohnert moves. Thus, combining parts (a) and (c), it follows that $$\{L_T(\tilde{r},c)~|~(\tilde{r},c)\in T,~\tilde{r}<r\}=\{L_{\tilde{T}}(\tilde{r},c)~|~(\tilde{r},c)\in \tilde{T},~\tilde{r}<r\}$$ for all $\tilde{T}\prec T$, i.e., the set of labels of non-ghost cells lying below a ghost cell is preserved by $K$-Kohnert moves. The result follows.

(f) Considering the definition of ghost move and combining parts (a) and (c), it follows that $$\{L_T(\tilde{r},c)~|~(\tilde{r},c)\in T,~\tilde{r}\le r\}=\{L_{\widehat{T}}(\tilde{r},c)~|~(\tilde{r},c)\in \widehat{T},~\tilde{r}\le r-1\}.$$ Thus, applying (c) once again, we have that $$L_T(r,c)=\max\{L_T(\tilde{r},c)\in T~|~(\tilde{r},c)\in T,~\tilde{r}\le r\}=\max\{L_{\widehat{T}}(\tilde{r},c)~|~(\tilde{r},c)\in \widehat{T},~\tilde{r}\le r-1\}=L_{\widehat{T}}(r,c),$$ as desired.

(g) Combining (e) and (f), it follows that $\langle r,c\rangle$ receives its labeling from a non-ghost cell occupying position $(r,c)$. Thus, applying part (b), the result follows.


(h) Considering the definition of $K$-Kohnert move, it follows that there must exist $T_1,T_2,T_3,T_4\in KKD(D)$ such that 
\begin{itemize}
    \item $T\preceq T_4\precdot T_3\prec T_2\precdot T_1\prec D$,
    \item $(r,c_1)\in T_1$ and $\langle r,c_2\rangle\notin T_1$,
    \item there exists $\widehat{r}_1$ such that $T_2=\mathcal{G}(T_1,r)=T_1\ldownarrow^{(r,c_1)}_{(\widehat{r}_1,c_1)}\cup\{\langle r,c_1\rangle\}$,
    \item $(r,c_2)\in T_3$, and 
    \item there exists $\widehat{r}_2$ such that $T_4=\mathcal{G}(T_3,r)=T_3\ldownarrow^{(r,c_2)}_{(\widehat{r}_2,c_2)}\cup\{\langle r,c_2\rangle\}$.
\end{itemize}
Thus, for all $\tilde{T}\in KKD(D)$ satisfying $T\preceq \tilde{T}\preceq T_2$, we have $\langle r,c_1\rangle\in \tilde{T}$ and, considering part (f), $L_{\tilde{T}}(r,c_1)=L_{T_1}(r,c_1)$; note that this implies that there exists $\tilde{r}<r$ such that $(\tilde{r},c_1)\in\tilde{T}$ and $L_{\tilde{T}}(\tilde{r},c_1)=L_{T_1}(r,c_1)$. Consequently, applying parts (a) and (d), it follows that the non-ghost cells with label $L_{T_1}(r,c_1)$ in columns $\tilde{c}\ge c_1$ must occupy rows $<r$ in $\tilde{T}$. In particular, there exists $r^*<r$ such that $(r^*,c_2)\in T_3$ and $L_{T_3}(r^*,c_2)=L_{T_1}(r,c_1)$. Therefore, applying parts (c) and (f), we find that $$L_T(r,c_2)=L_{T_4}(r,c_2)=L_{T_3}(r,c_2)>L_{T_3}(r^*,c_2)=L_{T_1}(r,c_1)=L_T(r,c_1),$$ as desired.

(i) Considering part (c), we have that $$L_T(r,c_1)=\max\{L_T(\tilde{r},c_1)~|~(\tilde{r},c_1)\in T,~\tilde{r}<r\}.$$ Consequently, it must be the case that if $(\tilde{r},c_1)\in T$ with $L_T(\tilde{r},c_1)=L_T(r+1,c_2)>L_T(r,c_1)$, then $\tilde{r}>r$. Applying part (d), it follows that if $(\tilde{r},\tilde{c})\in T$ with $L_T(\tilde{r},\tilde{c})=L_T(r+1,c_2)$ for $\tilde{c}\le c_1$, then $\tilde{r}>r$. Thus, considering our labeling of ghost cells, if $\langle r+1,\tilde{c}\rangle\in T$ for $\tilde{c}\le c_1$, then $L_T(r+1,\tilde{c})\neq L_T(r+1,c_2)$. The result follows.
\end{proof}

Lemma~\ref{lem:colorp} in hand, we can now prove Theorem~\ref{thm:lockmain}.

\begin{proof}[Proof of Theorem~\ref{thm:lockmain}]
    Let $M$ denote the number of snowflakes in $\mathsf{snow}^*(D)$ and $R=\{r_i~|~i\in[n],~r_i>i\}$. Assume that row $r_i$ of $D$ contains $m_i$ cells for $i\in [n]$ and set $N=\max\{m_i~|~i\in[n]\}$. Note that if $R=\emptyset$, then the cells of $D$ occupy rows 1 through $n$ with the rightmost cell of row $r\in [n]$ being $(r,N)$. Consequently, if $R=\emptyset$, then $KKD(D)=\{D\}$ and it is straightforward to verify that the result holds in this case. So, assume that $R\neq \emptyset$. Then there exists $\ell$ satisfying $1\le\ell\le n$ for which $R=\{r_i\}_{i=\ell}^n$ where $r_i<r_{i+1}$ for $\ell\le i<n$ when $n-\ell>0$. We refer to those diagrams $D$ for which all rows of $R$ contain dark clouds in $\mathsf{snow}(D)$ as type I diagrams and those for which there is a unique $r_k\in R$ such that row $r_k$ of $\mathsf{snow}(D)$ contains no dark cloud as type II. Note that, considering the definition of $\mathsf{snow}^*(D)$, if $D$ is a type I diagram, then $$M=\sum_{i=1}^n(r_i-i);$$ while if $D$ is a type II diagram with row $r_k\in R$ of $\mathsf{snow}(D)$ containing no dark cloud, then $$M=\sum_{k\neq i\in[n]}(r_i-i).$$

    First, we show that $M\le \mathrm{MaxG}(D)$. To do so, we provide an algorithm for generating a diagram $T\in KKD(D)$ with exactly $M$ ghost cells. Note that if $D$ is a type I diagram, then $m_i\ge n-i+1$ for $\ell\le i\le n$; while if $D$ is a type II diagram with row $r_k\in R$ of $\mathsf{snow}(D)$ containing no dark cloud, then $m_i\ge n-i+1$ for $k<i\le n$, $m_k<n-k+1$, and $m_i\ge n-i$ for $\ell\le i<k$.
    We form $T$ from $D$ as follows. For $i=\ell$ up to $n$ in increasing order,
    \begin{itemize}
        \item[\textup{(i)}] if there exists $k$ such that $i<k\le n$ and row $r_k$ of $\mathsf{snow}(D)$ contains no dark cloud, then apply in succession $n-i-1$ Kohnert moves followed by a single ghost move at rows $r_i$ down to $i+1$ in decreasing order; if row $r_i$ of $\mathsf{snow}(D)$ contains no dark cloud, then apply in succession $\min\{m_i,n-i\}$ Kohnert moves at rows $r_i$ down to $i+1$ in decreasing order; and if there exists no $k$ such that $i< k\le n$ and row $r_k$ of $\mathsf{snow}(D)$ contains no dark cloud, then apply in succession $n-i$ Kohnert moves followed by a single ghost move at rows $r_i$ down to $i+1$ in decreasing order.
    \end{itemize}
      Let $T_i$ denote the diagram formed after step $(i)$ of the procedure above for $\ell\le i\le n$; note that $T=T_n$. For a type I diagram $D$, we have that 
    \begin{itemize}
        \item for $\ell\le i<n$, $\langle r_i-j,N-n+i\rangle\in T_i$ with $L_{T_i}(r_i-j,N-n+i)=r_i$ for $0\le j\le r_i-i-1$, and $T_i\cap \{(r,c),\langle r,c\rangle~|~i+1\le r\le r_{i+1},~c>N-n+i\}=\emptyset$;
        \item and for $i=n$, $\langle r_n-j,N\rangle\in T_n$ with $L_{T_n}(r_n-j,N)=r_n$ for $0\le j\le r_n-n-1$.
    \end{itemize}
    Consequently, in this case, the diagram $T$ contains $\sum_{i=1}^n(r_i-i)$ ghost cells. Similarly, for a type II diagram $D$ with row $r_k\in R$ containing no dark cloud in $\mathsf{snow}(D)$, we have that
    \begin{itemize}
        \item for $\ell\le i<k$, $\langle r_i-j,N-n+i+1\rangle\in T_i$ with $L_{T_i}(r_i-j,N-n+i+1)=r_i$ for $0\le j\le r_i-i-1$, and $T_i\cap \{(r,c),\langle r,c\rangle~|~i+1\le r\le r_{i+1},~c>N-n+i+1\}=\emptyset$;
        \item for $i=k$, $T_k\cap \{(r,c),\langle r,c\rangle~|~k+1\le r\le r_{k+1},~c>N-n+k+1\}=\emptyset$;
        \item for $k<i<n$, $\langle r_i-j,N-n+i\rangle\in T_i$ with $L_{T_i}(r_i-j,N-n+i)=r_i$ for $0\le j\le r_i-i-1$, and $T_i\cap \{(r,c),\langle r,c\rangle~|~i+1\le r\le r_{i+1},~c>N-n+i\}=\emptyset$;
        \item and for $i=n$, $\langle r_n-j,N\rangle\in T_n$ with $L_{T_n}(r_n-j,N)=r_n$ for $0\le j\le r_n-n-1$.
    \end{itemize}
    Consequently, in this case, the diagram $T$ contains $\sum_{k\neq i\in[n]}(r_i-i)$ ghost cells. Thus, for either type, it follows that $M\le \mathrm{MaxG}(D)$.   

    Next, we show that $M\ge \mathrm{MaxG}(D)$. Considering the definition of our labeling along with Lemma~\ref{lem:colorp} (a), for all $T\in KKD(D)$ and $i\in [n]$, there exists $(\widehat{r}_i,N)\in T$ such that $L_T(\widehat{r}_i,N)=r_i$. Moreover, applying Lemma~\ref{lem:colorp} (c), if $r_i<r_j$ for $i,j\in [n]$, then $\widehat{r}_i<\widehat{r}_j$. Thus, it follows that $\widehat{r}_i\ge i$. Consequently, applying Lemma~\ref{lem:colorp} (d), if $(r,c)\in T$ satisfies $L_T(r,c)=r_i$, then $r\ge i$. Considering our definition for the labels of ghost cells, it follows that if $\langle r,c\rangle\in T$ satisfies $L_T(r,c)=r_i$, then $r>i$; that is, for all diagrams of $T\in KKD(D)$, ghost cells with label $r_i$ must lie strictly above row $i$. Therefore, applying Lemma~\ref{lem:colorp} (g) and (h), we have that $(\ast)$ for all $T\in KKD(D)$ and $i\in [n]$ there is at most one ghost cell labeled by $r_i$ in rows $i+1$ through $r_i$ of $T$, and no ghost cells with label $r_i$ outside of these rows; that is, $\mathrm{MaxG}(D)\le \sum_{i=1}^n(r_i-i)$. Thus, we have $M\ge \mathrm{MaxG}(D)$ when $D$ is a type I diagram.

    It remains to show that $M\ge \mathrm{MaxG}(D)$ when $D$ is a type II diagram. Assume that row $r_k\in R$ of $\mathsf{snow}(D)$ contains no dark cloud. Note that, considering $(\ast)$, the value $$M=\sum_{k\neq i\in [n]}r_i-i+1$$ is equal to to the maximum number of ghost cells with labels $r_i$ for $k\neq i\in [n]$ that can be contained in any $T\in KKD(D)$. Thus, to prove that $M\ge \mathrm{MaxG}(D)$, it suffices to show that if $T\in KKD(D)$ contains a ghost cell labeled by $r_k$, then each such cell can be uniquely paired with a ghost cell labeled by $r_i$ for $k\neq i\in [n]$ which could, but does not occur in $T$. In particular, keeping in mind that, by $(\ast)$, ghost cells labeled by $r_i$ in $T$ can only occur in rows $j$ for $i< j\le r_i$, we show that each ghost cell labeled by $r_k$ in $T$ can be uniquely paired with a tuple $(r_i,j)$ for $k\neq i\in [n]$ and $i<j\le r_i$ such that row $j$ of $T$ contains no ghost cell with label $r_i$. 

    Since row $r_k$ is the unique row of $D$ containing no dark clouds in $\mathsf{snow}(D)$, it follows that there are at least $m_k$ nonempty rows strictly above row $r_k$ in $D$, i.e., $k+m_k\le n$; note that since $m_k\neq0$, it follows that $k<n$. Assume that $\langle r,j\rangle\in T$ with $k<r\le r_k$ satisfies $L_T(r,j)=r_k$. Note that, applying Lemma~\ref{lem:colorp} (a), we must have $N-m_k<j\le N$. If $j=N$, then, considering Lemma~\ref{lem:colorp} (i), there exists no ghost cell in row $r+1$ of $T$ labeled by $r_{k+1}$. Thus, since $k+1\in [n]$ and $k+1<r+1\le r_k+1\le r_{k+1}$, we can pair $\langle r,j\rangle$ with $(r_{k+1},r+1)$. So, assume that $j<N$. In this case, if there exists $t$ such that $0<t\le N-j<m_k\le n-k$ and row $r+t$ of $T$ contains no ghost cell labeled by $r_{k+t}$, then let $t^*$ be the least such $t$. Since $$1<k+t^*\le k+N-j\le k+N-(N-m_k+1)=k+m_k-1\le n,$$ i.e., $k+t^*\in [n]$, and $$k+t^*<r+t^*\le r_k+t^*\le r_{k+t^*},$$ we can pair $\langle r,j\rangle$ with $(r_{k+t^*},r+t^*)$. On the other hand, if for all $t$ satisfying $0<t\le N-j<m_k\le n-k$, row $r+t$ of $T$ contains a ghost cell labeled by $r_{k+t}$, then, applying Lemma~\ref{lem:colorp} (i), it follows that $\langle r+t,j+t\rangle\in T$ is labeled by $r_{k+t}$ for $0<t\le N-j$; but then, once again applying Lemma~\ref{lem:colorp} (i), we may conclude that there is no ghost cell labeled by $r_{k+N-j+1}$ in row $r+N-j+1$ of $T$. Since $$1<k+N-j+1\le k+N-(N-m_k+1)+1=k+m_k\le n,$$ i.e., $k+N-j+1\in [n]$, and $$k+N-j+1<r+N-j+1\le r_k+N-j+1\le r_{k+N-j+1},$$ it follows that, in this case, we can pair $\langle r,j\rangle$ with $(r_{k+N-j+1},r+N-j+1)$. As there can only be at most one ghost cell labeled by $r_k$ in each row of $T$, it follows that the pairing above matches distinct ghost cells labeled by $r_k$ in $T$ with distinct missing ghost cells labeled by $r_i$ for $k\neq i\in [n]$. The result follows.
\end{proof}

It would be interesting to consider how Theorem~\ref{thm:lockmain} could be extended to apply to all lock diagrams. At present, it is unclear to the authors how one might utilize the approach via labeling to understand $\mathrm{MaxG}(D)$ for arbitrary lock diagrams $D$.

\begin{remark}
    The labeling defined above for diagrams of $KKD(D)$ when $D$ is a lock diagram can be extended to more general diagrams. For an arbitrary diagram $D$, we define a recursive labeling of the diagrams in $KKD(D)$ starting from $D$ as follows. For $T\in KKD(D)$, denoting the label of $(r,c)$ or $\langle r,c\rangle\in T$ by $L_T(r,c)$, we set $L_D(r,c)=r$ for all $(r,c)\in D$, i.e., cells of $D$ are labeled by the rows that they occupy. Now, assume that $T\in KKD(D)$ is such that $T$ has yet to be labeled and there exists a labeled $T^*\in KKD(D)$ for which either $$T=\mathcal{K}(T^*,r^*)=T^*\ldownarrow^{(r^*,c^*)}_{(r^*-k,c^*)}\quad\quad\text{or}\quad\quad \mathcal{G}(T^*,r^*)=T^*\ldownarrow^{(r^*,c^*)}_{(r^*-k,c^*)}\cup \{\langle r^*,c^*\rangle\}$$ for $k\ge 1$. Then 
    \begin{itemize}
        \item[\textup{(i)}] $L_T(r^*-j,c^*)=L_{T^*}(r^*-j+1,c^*)$ for $1\le j\le k$;
        \item[\textup{(ii)}] if $L_{T^*}(r^*-j,c)\le L_{T^*}(r^*-j,c^*)$ for $c>c^*$ and $1\le j<k$, then $L_{T}(r^*-j,c)=L_{T^*}(r^*-j+1,c^*)$;
        \item[\textup{(iii)}] $L_T(r^*,c^*)=L_{T^*}(r^*,c^*)$ in the case $T=\mathcal{G}(T^*,r^*)$; and
        \item[\textup{(iv)}] $L_T(r,c)=L_{T^*}(r,c)$ for all remaining $(r,c)$ and $\langle r,c\rangle\in T^*\cap T$.
    \end{itemize}
It can be shown that when $D$ is a lock diagram, the labeling described here is the same as the one extended from Theorem~\ref{thm:label}. For general diagrams $D$, the labeling of $T\in KKD(D)$ given via the more general approach outlined above can depend on the $T^*$ chosen. Finally, it is worth noting that many of the results of this section concerning the distribution of labeled ghost cells can be established more generally using the labeling described here. Full details concerning the labeling defined here are omitted for the sake of brevity.
\end{remark}

\section{Greedy Approach}\label{sec:greedy}

In this section, given a diagram $D$ that contains no ghost cells, we consider the problem of computing the maximum number of ghost cells contained in a diagram $T\in KKD(D)$ formed by applying sequences of only ghost moves to $D$. Throughout this section, given a diagram $D$, let $GKD(D)$ denote the collection of diagrams consisting of $D$ along with all those diagrams which can be formed from $D$ by applying sequences of only ghost moves. Moreover, let $$\widehat{\mathrm{MaxG}}(D)=\max\{|G(\tilde{D})|~|~\tilde{D}\in GKD(D)\}.$$ Unlike in the case of $\mathrm{MaxG}(D)$, we are able to establish a result in the spirit of Theorem~\ref{thm:JP} which applies to the computation of $\widehat{\mathrm{MaxG}}(D)$ when $D$ is an arbitrary diagram that contains no ghost cells. To state the main result of this section, we require a diagram $\widehat{\mathsf{snow}}(D)$ analogous to $\mathsf{snow}(D)$.

Given an arbitrary diagram $D$ with nonempty columns $\{c_i\}_{i=1}^m$, where $c_i<c_{i+1}$ for $i\in [m-1]$ when $m>1$, let $$R_i=\{r~|~(r,c_i)\in D\}\quad\text{and}\quad F_i=\left\{r\in R_i~|~r\in \bigcup_{j=i+1}^nR_j\right\}$$ for $i\in [m]$; that is, $R_i$ consists of the rows containing non-ghost cells in column $c_i$ of $D$ and $F_i$ consists of the rows $r\in R_i$ for which the cell $(r,c_i)\in D$ is not rightmost in row $c_i$ of $D$. Form $\widehat{\mathsf{snow}}(D)$ as follows.
\begin{enumerate}
    \item First, we label the cells of $D$. Denoting the label of a cell $(r,c)\in D$ by $\widehat{L}_D(r,c)$, set $\widehat{L}_D(r,c_n)=1$ for all $(r,c_n)\in D$. As for the remaining cells of $D$, for each $i\in [n-1]$,
    \begin{itemize}
        \item set $\widehat{L}_D(r,c_i)=r$ for all $(r,c)\in D$ with $r\in F_i$; then,
        \item in decreasing order of $r\in R_i\backslash F_i$, set $\widehat{L}_D(r,c_i)$ to be either
        \begin{enumerate}
            \item[$(a)$] the maximal $r^*\in \bigcup_{j=i+1}^nR_j$ such that $r^*<r$ and no $(\tilde{r},c_i)\in D$ for $\tilde{r}>r$ has $\widehat{L}_D(\tilde{r},c_i)=r^*$ or
            \item[$(b)$] 1 if no such value exists.
        \end{enumerate}
    \end{itemize}
    \item For each empty position $(r,c)\notin D$ lying below some $(\widehat{r},c)\in D$ with $\widehat{r}>r$, draw a snowflake $\ast$ in position $(r,c)$ if there exists $(\tilde{r},c)\in D$ with $\tilde{r}>r$ and $\widehat{L}_D(\tilde{r},c)\le r$; otherwise, do nothing.
\end{enumerate}

\begin{example}
For the diagram $D$ of Figure~\ref{fig:Fex} \textup{(a)}, we illustrate $\widehat{\mathsf{snow}}(D)$ in Figure~\ref{fig:Fex} \textup{(b)}. Considering Theorem~\ref{thm:gmain} below, it follows that $\widehat{\mathrm{MaxG}}(D)=8$.
\begin{figure}[H]
    \centering
    $$\scalebox{0.8}{\begin{tikzpicture}[scale=0.65]
        \node at (0.5, 1.5) {$\cdot$};
        \node at (1.5, 1.5) {$\cdot$};
        \node at (3.5, 1.5) {$\cdot$};
        
        \node at (2.5, 2.5) {$\cdot$};
        \node at (5.5, 2.5) {$\cdot$};
        \node at (7.5, 2.5) {$\cdot$};

        \node at (1.5, 3.5) {$\cdot$};
        \node at (3.5, 3.5) {$\cdot$};
        \node at (4.5, 3.5) {$\cdot$};

        \node at (0.5, 4.5) {$\cdot$};
        \node at (1.5, 4.5) {$\cdot$};
        \node at (2.5, 4.5) {$\cdot$};
        \node at (3.5, 4.5) {$\cdot$};

        \node at (5.5, 5.5) {$\cdot$};
        \node at (6.5, 5.5) {$\cdot$};
        \node at (7.5, 5.5) {$\cdot$};

        \node at (1.5, 6.5) {$\cdot$};
        \node at (2.5, 6.5) {$\cdot$};
        
        \draw (0,7.5)--(0,0)--(8.5,0);
        
        \draw (0,1)--(2,1)--(2,2)--(0,2)--(0,1);
        \draw (1,1)--(1,2);
        \draw (3,1)--(4,1)--(4,2)--(3,2)--(3,1);
        
        \draw (2,2)--(3,2)--(3,3)--(2,3)--(2,2);
        \draw (5,2)--(6,2)--(6,3)--(5,3)--(5,2);
        \draw (7,2)--(8,2)--(8,3)--(7,3)--(7,2);

        \draw (1,3)--(2,3)--(2,4)--(1,4)--(1,3);

        \draw (3,3)--(5,3)--(5,4)--(3,4)--(3,3);
        \draw (4,3)--(4,4);

        \draw (0,4)--(4,4)--(4,5)--(0,5);
        \draw (1,4)--(1,5);
        \draw (2,4)--(2,5);
        \draw (3,4)--(3,5);

        \draw (5,5)--(8,5)--(8,6)--(5,6)--(5,5);
        \draw (6,5)--(6,6);
        \draw (7,5)--(7,6);

        \draw (1,6)--(3,6)--(3,7)--(1,7)--(1,6);
        \draw (2,6)--(2,7);
        \node at (4,-1) {\Large (a)};
\end{tikzpicture}}\quad\quad\quad\quad\scalebox{0.8}{\begin{tikzpicture}[scale=0.65]
        \node at (0.5, 1.5) {$2$};
        \node at (1.5, 1.5) {$2$};
        \node at (3.5, 1.5) {$1$};
        \node at (3.5, 0.5) {$\ast$};
        
        \node at (2.5, 2.5) {$3$};
        \node at (5.5, 2.5) {$3$};
        \node at (7.5, 2.5) {$1$};

        \node at (1.5, 3.5) {$4$};
        \node at (3.5, 3.5) {$4$};
        \node at (4.5, 3.5) {$3$};
        \node at (4.5, 2.5) {$\ast$};
        \node at (3.5, 2.5) {$\ast$};

        \node at (0.5, 4.5) {$5$};
        \node at (1.5, 4.5) {$5$};
        \node at (2.5, 4.5) {$5$};
        \node at (3.5, 4.5) {$3$};

        \node at (5.5, 5.5) {$6$};
        \node at (6.5, 5.5) {$6$};
        \node at (7.5, 5.5) {$1$};
        \node at (7.5, 4.5) {$\ast$};
        \node at (7.5, 3.5) {$\ast$};
        \node at (7.5, 1.5) {$\ast$};
        \node at (7.5, 0.5) {$\ast$};

        \node at (1.5, 6.5) {$7$};
        \node at (2.5, 6.5) {$6$};
        \node at (2.5, 5.5) {$\ast$};
        
        \draw (0,7.5)--(0,0)--(8.5,0);
        
        \draw (0,1)--(2,1)--(2,2)--(0,2)--(0,1);
        \draw (1,1)--(1,2);
        \draw (3,1)--(4,1)--(4,2)--(3,2)--(3,1);
        
        \draw (2,2)--(3,2)--(3,3)--(2,3)--(2,2);
        \draw (5,2)--(6,2)--(6,3)--(5,3)--(5,2);
        \draw (7,2)--(8,2)--(8,3)--(7,3)--(7,2);

        \draw (1,3)--(2,3)--(2,4)--(1,4)--(1,3);

        \draw (3,3)--(5,3)--(5,4)--(3,4)--(3,3);
        \draw (4,3)--(4,4);

        \draw (0,4)--(4,4)--(4,5)--(0,5);
        \draw (1,4)--(1,5);
        \draw (2,4)--(2,5);
        \draw (3,4)--(3,5);

        \draw (5,5)--(8,5)--(8,6)--(5,6)--(5,5);
        \draw (6,5)--(6,6);
        \draw (7,5)--(7,6);

        \draw (1,6)--(3,6)--(3,7)--(1,7)--(1,6);
        \draw (2,6)--(2,7);
        \node at (4,-1) {\Large (b)};
\end{tikzpicture}}$$
    \caption{$\widehat{\mathsf{snow}}(D)$}
    \label{fig:Fex}
\end{figure}
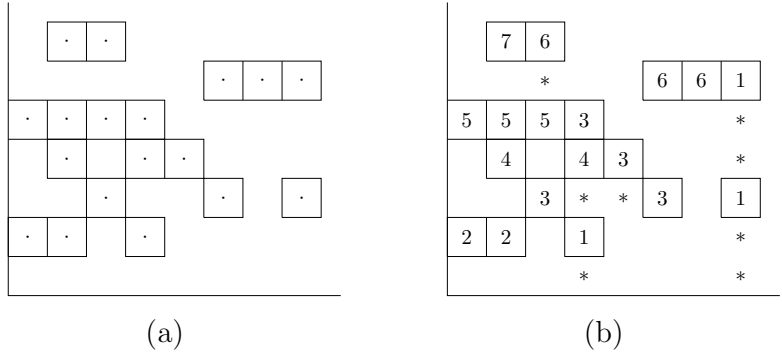
\end{example}

Letting $\widehat{\text{sf}}(D)$ denote the number of snowflakes in $\widehat{\mathsf{snow}}(D)$, the main result of this section is as follows.

\begin{theorem}\label{thm:gmain}
If $D$ is a diagram that contains no ghost cells, then $\widehat{\mathrm{MaxG}}(D)=\widehat{\textup{sf}}(D)$. 
\end{theorem}

As the corresponding notation will prove to be useful ongoing, we define a poset structure on $GKD(D)$. In particular, given a diagram $D$ that contains no ghost cells, for $D_1,D_2\in GKD(D)$ we say $D_2\prec_gD_1$ provided that $D_2$ can be obtained from $D_1$ by applying some sequence of ghost moves; note that this is the poset obtained by restricting the poset structure on $KKD(D)$ introduced in Section~\ref{sec:extJP} to the diagrams of $GKD(D)$. If $D_1$ covers $D_2$ in this ordering, then we write $D_2\precdot_g D_1$. Note that if $D_2\precdot_gD_1$, then $D_2$ can be formed from $D_1$ by applying a single ghost move. The poset structure on $GKD(D)$ defined here is considered in \textup{\textbf{\cite{GK}}}.

The following lemma allows us to provide an interpretation for the labels of cells in $\widehat{\mathsf{snow}}(D)$.

\begin{lemma}\label{lem:block}
    Let $D$ be a diagram that contains no ghost cells. If $(r,c)\in D$ and there exists $T\in KKD(D)$ such that $(r,\tilde{c})\in T$ with $\tilde{c}<c$, then $(r,\tilde{c})\in \tilde{T}$ for all $\tilde{T}\prec_g T$.
\end{lemma}
\begin{proof}
    Note that since $(r,c)\in D$, it follows that for all $T\in KKD(D)$ either $(r,c)\in T$ or $\langle r,c\rangle\in T$. Consequently, for any $T\in KKD(D)$ and $\tilde{c}<c$, if $(r,\tilde{c})\in T$, then $(r,\tilde{c})$ is not rightmost in row $r$ of $T$. The result follows.
\end{proof}

\begin{remark}
    Take $(r,c)\in D$ which has label $\widehat{L}(r,c)=r^*$ in $\widehat{\mathsf{snow}}(D)$. Label the cell $(r,c)$ by $r^*$ in $D$ and assume that the label is retained through the application of ghost moves; that is, if $(\tilde{r},c)\in T_1\prec_g D$ is labeled $r^*$ and $T_2=\mathcal{G}(T_1,\widehat{r}),$ then either
    \begin{enumerate}
        \item[$(1)$] $(\tilde{r},c)\in T_2$ and $(\tilde{r},c)$ is the unique cell labeled by $r^*$ in $T_2$ or
        \item[$(2)$] there exists $1\le j\le \tilde{r}-1$ such that $$T_2=T_1\ldownarrow^{(\tilde{r},c)}_{(\tilde{r}-j,c)}\cup\{\langle \tilde{r},c\rangle\}$$ in which case the unique cell labeled by $r^*$ in $T_2$ is $(\tilde{r}-j,c)$.
    \end{enumerate}
    Considering Lemma~\ref{lem:block}, the label $r^*$ of the cell $(r,c)\in D$ corresponds to a row weakly below in which, if ever occupied by, the cell will remain after the application of any sequence of ghost moves; note that Lemma~\ref{lem:block} is not required for this conclusion when $r^*=1$.
\end{remark}

\begin{proof}[Proof of Theorem~\ref{thm:gmain}]
    First, we show that $\widehat{\mathrm{MaxG}}(D)\le \widehat{\text{sf}}(D)$. To this end, let $$D_1=\{(r,c)~|~\exists~T\in KKD(D)~\text{such that}~(r,c)\in T\}$$ and $D_2$ denote the set of positions $(r,c)$ in which $\widehat{\mathsf{snow}}(D)$ contains either a cell or a $\ast$. Note that for $T\in GKD(D)$, since ghost cells can only occupy positions that non-ghost cells previously occupied, $D_1$ can be described as the set of positions that can be occupied by a cell in $T$, ghost or non-ghost. If we can show that $D_1\subseteq D_2$, then it will follow that $\widehat{\mathrm{MaxG}}(D)\le \widehat{\text{sf}}(D)$. To see this, note that $D_1\subseteq D_2$ implies that the number of cells, both ghost and non-ghost, contained in $T\in GKD(D)$ is bounded above by $$|D_2|=|\{(r,c)~|~(r,c)\in D\}|+\widehat{\text{sf}}(D).$$ Thus, since every $T\in GKD(D)$ contains exactly $|\{(r,c)~|~(r,c)\in D\}|$ non-ghost cells, if $D_1\subseteq D_2$, then it will follow that the number of ghost cells contained in $T\in GKD(D)$ is bounded above by $$|D_2|-|\{(r,c)~|~(r,c)\in D\}|=\widehat{\text{sf}}(D),$$ as desired. Now, to establish that $D_1\subseteq D_2$, we show that if $(r,c)$ is an empty position of $\widehat{\mathsf{snow}}(D)$, i.e., contains no cell or $\ast$, then $(r,c)\notin T$ for all $T\in KKD(D)$.

    Assume that $(r^\circ,c^\circ)$ corresponds to an empty position of $\widehat{\mathsf{snow}}(D)$. Evidently, if there exists no $r>r^\circ$ such that $(r,c^\circ)\in D$, then $(r^\circ,c^\circ)\notin T$ for all $T\in GKD(D)$. So, in addition, assume that $$\emptyset\neq\{r>r^\circ~|~(r,c^\circ)\in D\}=U.$$ Since position $(r^\circ,c^\circ)$ of $\widehat{\mathsf{snow}}(D)$ is empty, it follows that $\widehat{L}(r,c^\circ)>r^\circ$ for all $r\in  U$. For a contradiction, assume that there exists $T\in GKD(D)$ for which $(r^\circ,c^\circ)\in T$. We define a labeling of the non-ghost cells of $T$ iteratively starting from $D$ and using the labeling of cells in $\widehat{\mathsf{snow}}(D)$ as follows. Assume that $T\in GKD(D)$ was formed from $D$ by applying ghost moves at rows $\{r_i\}_{i=1}^n$ in increasing order of $i$; that is, letting $D_0=D$ and $D_i=\mathcal{G}(D_{i-1},r_i)$ for $i\in [n]$, we have $D_{i}\prec_g D_{i-1}$ for $i\in [n]$ and $D_n=T$. Letting $L_{D_i}(r,c)$ denote the label of $(r,c)\in D_i$ for $0\le i\le n$, we have $L_{D_0}(r,c)=L_D(r,c)=\widehat{L}_D(r,c)$ for all $(r,c)\in D_0$, i.e., the cells of $D$ are labeled as in $\widehat{\mathsf{snow}}(D)$; and for $i\in [n]$, if $$D_i=D_{i-1}\ldownarrow^{(r_i,c_i)}_{(\widehat{r}_i,c_i)}\cup\{\langle r_i,c_i\rangle\},$$ then we define $L_{D_i}(\widehat{r}_i,c_i)=L_{D_{i-1}}(r_i,c_i)$ and $L_{D_i}(r,c)=L_{D_{i-1}}(r,c)$ for all $(r,c)\in D_i\cap D_{i-1}=D_{i-1}\backslash\{(r_i,c_i)\}$. Using this labeling of the non-ghost cells of $T$, we now construct a sequence of values $\{i_j\}_{j\ge 1}$ contained in $\mathbb{N}$ as follows.
\begin{enumerate}
        \item[] $i_1$: Define $i_1=L_T(r^\circ,c^\circ)$, i.e., $i_1$ is the label of the cell $(r^\circ,c^\circ)\in T$.
        \item[] $i_2$: Assume that the cell labeled by $i_1$ in column $c^\circ$ of $\widehat{\mathsf{snow}}(D)$ occupies row $r_{i_1}$. Note that $r_{i_1}>i_1>r^\circ$ considering the definition of our labeling and our assumptions on the position $(r^\circ,c^\circ)$. Let $$L_1=\{r~|~r^\circ<r< r_{i_1}~\text{and}~\exists~\tilde{c}>c^\circ~\text{such that}~(r,\tilde{c})\in D\},$$ i.e., $L_1$ consists of positions in column $c^\circ$ of $T$ between rows $r_{i_1}$ and $r^\circ$ for which an occupying cell of $T$ would not be rightmost in its row by Lemma~\ref{lem:block}; note that $L_1\neq\emptyset$ since $i_1\in L_1$. Keeping in mind our labeling of the cells in $D_i$ for $0\le i\le n$, since the cell labeled by $i_1$ occupies position $(r_{i_1},c^\circ)$ in $D_0=D$ and moves to position $(r^\circ,c^\circ)$ in $T=D_n$, it follows that for each $\tilde{r}\in L_1$ we have $(\tilde{r},c^\circ)\in T$. Moreover, since $i_1\in L_1$, $r^\circ\notin L_1$, and the cell labeled by $i_1$ in column $c^\circ$ of $T$ occupies position $(r^\circ,c^\circ)$, it follows that there exists $(\tilde{r},c^\circ)\in T$ with $\tilde{r}\in L_1$ and $r^*=L_T(\tilde{r},c^\circ)>\max L_1$. Let $i_2$ denote the maximal such value, i.e., $i_2=\max\{L_T(\tilde{r},c^\circ)~|~\tilde{r}\in L_1~\text{and}~(\tilde{r},c^\circ)\in T\}$. Note that $i_2>i_1$ by construction.
        \item[] $i_j$ for $j>2$: Assume the cell labeled by $i_{j-1}$ in column $c^\circ$ of $\widehat{\mathsf{snow}}(D)$ occupies row $r_{i_{j-1}}$. Note that $r_{i_{j-1}}>r^\circ$ by construction. Let $$L_{j-1}=\{r~|~r^\circ<r<r_{i_{j-1}}~\text{and}~\exists \tilde{c}>c^\circ~\text{such that}~(r,\tilde{c})\in D\},$$ i.e., $L_{j-1}$ consists of positions in column $c^\circ$ of $T$ between rows $r_{i_{j-1}}$ and $r^\circ$ for which an occupying cell of $T$ would not be rightmost in its row by Lemma~\ref{lem:block}. Note that, since $r_{i_{j-1}}>i_{j-1}>\max L_{j-2}$ by construction, it follows that $L_{j-2}\subsetneq L_{j-1}$. From the construction of $i_{j-1}$, it follows that for all $\tilde{r}\in L_{j-2}$ we have $(\tilde{r},c^\circ)\in T$. Moreover, keeping in mind our labeling of the cells in $D_i$ for $0\le i\le n$, since the cell labeled by $i_{j-1}$ occupies position $(r_{i_{j-1}},c^\circ)$ in $D_0=D$ and moves to a position $(\widehat{r},c^\circ)$ in $T$ with $\widehat{r}\in L_{j-2}$, it follows that for all $\tilde{r}\in L_{j-1}$ we have $(\tilde{r},c^\circ)\in T$. Thus, since $i_1\in L_{j-1}$, $r^\circ\notin L_1$, and the cell labeled by $i_1$ of $T$ occupies position $(r^\circ,c^\circ)$, it follows that there exists $(\tilde{r},c^\circ)\in T$ with $\tilde{r}\in L_{j-1}$ and $r^*=L_T(\tilde{r},c^\circ)>\max L_{j-1}$. Let $i_j$ denote the maximal such value, i.e., $i_j=\max\{L_T(\tilde{r},c^\circ)~|~\tilde{r}\in L_{j-1}~\text{and}~(\tilde{r},c^\circ)\in T\}$. Note that $i_j>i_{j-1}$ by construction.
\end{enumerate}
Intuitively, each $i_j$ is the largest label of a cell which is used to fill a position, allowing the cell labeled by $i_{j-1}$ to reach its final position in $T$. Note that this procedure can be performed indefinitely, forming an infinite increasing sequence. Since the collection of labels of cells in column $c^\circ$ is finite, this is a contradiction. Thus, $(r^\circ,c^\circ)\notin T$ for all $T\in GKD(D)$ and we may conclude that $D_1\subseteq D_2$. As noted above, it follows that $\widehat{\mathrm{MaxG}}(D)\le \widehat{\text{sf}}(D)$.

Finally, to show that $\widehat{\mathrm{MaxG}}(D)\ge \widehat{\text{sf}}(D)$, it suffices to show that there exists a sequence of ghost moves which can be applied to $D$ to form $T\in GKD(D)$ satisfying $|G(T)|=\widehat{\text{sf}}(D)$. To accomplish this, we show that for any nonempty column $c$ of $D$, there exists a sequence of diagrams $\{D_i\}_{i=0}^n\subset GKD(D)$ with $D_n=T$ such that
    \begin{itemize}
        \item if $n>0$, then for each $i\in [n]$ there exists $r_i$ such that $$D_i=\mathcal{G}(D_{i-1},r_i)=D_{i-1}\ldownarrow^{(r_i,c)}_{(\widehat{r}_i,c)}\cup\{\langle r_i,c\rangle\};$$ and
        \item the number of ghost cells in column $c$ of $T$ is equal to the number of snowflakes $\ast$ in column $c$ of $\widehat{\mathsf{snow}}(D)$;
    \end{itemize}
    that is, there exists a sequence of ghost moves which, when applied to $D$, only affects column $c$ and introduces the appropriate number of ghost cells. Our proof is by induction on $N=\max\{r~|~(r,c)\in R(D)\}$. Note that the base case is immediate as, if $N=1$, then column $c$ of $\widehat{\mathsf{snow}}(D)$ contains no $\ast$'s. Assume that the result holds for $N-1\ge 1$ and that $D$ is a diagram for which $N=\max\{r~|~(r,c)\in R(D)\}$. Moreover, assume that $$D_1=\mathcal{G}(D,N)=D\ldownarrow^{(N,c)}_{(\widehat{N},c)}\cup\{\langle N,c\rangle\}$$ with $N\neq \widehat{N}$; note that if $\widehat{N}=N$, then column $c$ of $\widehat{\mathsf{snow}}(D)$ contains no $\ast$'s and we are done. Considering the definition of $\widehat{\mathsf{snow}}(D)$, it follows that $\widehat{L}(N,c)\le \widehat{N}$ and $\widehat{\mathsf{snow}}(D)$ contains a $\ast$ in position $(\widehat{N},c)$ with any other $\ast$'s in column $c$ occurring below row $\widehat{N}$. Let $D^*=D_1\backslash\{\langle N,c\rangle\}$. If we can show that $\widehat{\mathsf{snow}}(D^*)$ contains $\ast$'s in the same positions $(r,c)$ with $r<\widehat{N}$ as $\widehat{\mathsf{snow}}(D)$, then the result will follow by induction. To see this, first note that our induction hypothesis applies to $D^*$. Thus, there exists a sequence of ghost moves which can be applied to $D^*$ affecting only cells in column $c$ and introducing the same number of ghost cells in column $c$ as $\ast$'s in column $c$ of $\widehat{\mathsf{snow}}(D^*)$. Let $T$ denote the diagram formed by applying the same sequence of ghost moves to $D_1$. By the definitions of $(N,c)\in D$ and $D^*$, we have that for $(\tilde{r},c)$ with $\tilde{r}\ge N$, either $(\tilde{r},c)\notin D^*$ or $(\tilde{r},c)\notin R(D^*)$. Consequently, any $\ast$'s in column $c$ of $\widehat{\mathsf{snow}}(D^*)$ occupy rows $\tilde{r}<\widehat{N}$ so that all ghost moves applied in forming $T$ from $D_1$ must correspond to rows $\tilde{r}<\widehat{N}$. Thus, since $D^*$ and $D_1$ match weakly below row $\widehat{N}$, it follows that the sequence of ghost moves applied to form $T$ from $D_1$ introduces ghost cells in the same rows in column $c$ as if the sequence were applied to $D^*$, and affects no cells of any other column; that is, if $\widehat{\mathsf{snow}}(D^*)$ contains $\ast$'s in the same positions $(r,c)$ with $r<\widehat{N}$ as $\widehat{\mathsf{snow}}(D)$, then $T$ would have the desired number of ghost cells in column $c$ and no ghost cells in any other column.

    To show that $\widehat{\mathsf{snow}}(D)$ and $\widehat{\mathsf{snow}}(D^*)$ contain $\ast$'s in the same positions $(r,c)$ with $r<\widehat{N}$, we consider how the labeling of the cells in column $c$ of $\widehat{\mathsf{snow}}(D^*)$ relates to that of those in $\widehat{\mathsf{snow}}(D)$. Note that for all $(\tilde{r},c)\in D$ satisfying $(\tilde{r},c)\notin R(D)$, we have $(\tilde{r},c)\notin R(D^*)$ so that $\widehat{L}_D(\tilde{r},c)=\widehat{L}_{D^*}(\tilde{r},c)$. Now, let $$R=\{\tilde{r}~|~\widehat{N}<\tilde{r}<N,~(\tilde{r},c)\in R(D)\}=\{\tilde{r}~|~\widehat{N}<\tilde{r}<N,~(\tilde{r},c)\in R(D^*)\}.$$ Moreover, let $L=\{\widehat{L}_D(\tilde{r},c)~|~\tilde{r}\in R\}$. If $R=\emptyset$, then it follows that $\widehat{L}_D(N,c)=\widehat{L}_{D^*}(\widehat{N},c)$ and $\widehat{L}_D(r,c)=\widehat{L}_{D^*}(r,c)$ for all $(r,c)\in \{(\tilde{r},c)\in D~|~\tilde{r}\neq N\}=\{(\tilde{r},c)\in D^*~|~\tilde{r}\neq \widehat{N}\}$; note that, as a consequence, the number of $\ast$'s occurring below row $\widehat{N}$ in $\widehat{\mathsf{snow}}(D^*)$ is equal to that of $\widehat{\mathsf{snow}}(D)$. On the other hand, if $R\neq \emptyset$, then assume that $R=\{r_i\}_{i=1}^n$ and $L=\{\ell_i\}_{i=1}^n$ with $r_i>r_{i+1}$ for $i\in [n-1]$ in the case $n>1$; note that if $n>1$, then $\ell_i>\ell_{i+1}$ for $i\in [n-1]$. By construction of $\widehat{\mathsf{snow}}(D)$ we have that $\widehat{N}\ge \widehat{L}_D(N,c)>\ell_1$. Thus, it follows that $\widehat{L}_{D^*}(r_1,c)=\widehat{L}_D(N,c)$, $\widehat{L}_{D^*}(\widehat{N},c)=\widehat{L}_D(r_n,c)$, and, in the case that $n>1$, $\widehat{L}_{D^*}(r_{i+1},c)=\widehat{L}_D(r_{i},c)$ for $i\in [n-1]$. Since the set of labels of cells in rows $\widehat{N}\le r<N$ of column $c$ in $\widehat{\mathsf{snow}}(D^*)$ is equal to that for cells in rows $\widehat{N}< r\le N$ of column $c$ in $\widehat{\mathsf{snow}}(D)$ and $$D^*\cap\{(\tilde{r},c)~|~\tilde{r}<\widehat{N}\}=D\cap\{(\tilde{r},c)~|~\tilde{r}<\widehat{N}\},$$ it follows that all cells below row $\widehat{N}$ in column $c$ of $D^*$ have the same label in $\widehat{\mathsf{snow}}(D^*)$ as those in $\widehat{\mathsf{snow}}(D)$. Consequently, considering the definition of $\widehat{\mathsf{snow}}(-)$, it follows that $\widehat{\mathsf{snow}}(D^*)$ and $\widehat{\mathsf{snow}}(D)$ contain the same number of $\ast$'s in positions $(r,c)$ with $r<N^*$. As noted above, the result follows.
\end{proof}

Since we now have a method for computing $\widehat{\mathrm{MaxG}}(D)$ for any diagram $D$, it is natural to wonder if this could be applied to solve our original problem of computing $\mathrm{MaxG}(D)$. Unfortunately, it is not difficult to produce example diagrams $D$ for which $\widehat{\mathrm{MaxG}}(D)<\mathrm{MaxG}(D)$, see Example~\ref{ex:greedyno} below.

\begin{example}\label{ex:greedyno}
    Take $D$ to be the diagram illustrated in Figure~\ref{fig:greedyfail} below. Computing, one finds that $\widehat{\mathrm{MaxG}}(D)=1<2=\mathrm{MaxG}(D)$. 
    \begin{figure}[H]
        \centering
        $$\scalebox{0.8}{\begin{tikzpicture}[scale=0.65]
  \node at (1.5, 2.5) {$\cdot$};
  \node at (0.5, 1.5) {$\cdot$};
  \node at (1.5, 1.5) {$\cdot$};
  \draw (0,4)--(0,0)--(3,0);
  \draw (0,1)--(2,1)--(2,3)--(1,3)--(1,2)--(0,2);
  \draw (1,1)--(1,2)--(2,2);
\end{tikzpicture}}$$
        \caption{Diagram for which greedy approach fails}
        \label{fig:greedyfail}
    \end{figure}
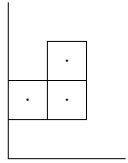
\end{example}

To end this section, we show how in the greedy case one can often simplify computations by reducing to a smaller diagram. In particular, for any diagram $D$, we define a reduction $f_D(D)\subseteq D$ for which $\widehat{\mathrm{MaxG}}(D)=\widehat{\mathrm{MaxG}}(f_D(D))$. It would be interesting to consider whether a similar reduction could be defined corresponding to $\mathrm{MaxG}(D)$.

\begin{definition}[Reduction Function]
Given a diagram $D$ that contains no ghost cells, define $$S(D)=\{(r,c)\in D~|~(r,c)\notin R(D)~\text{and}~ (r^*,c)\notin R(D)~\forall~r^*>r\}$$ and $f_D(\tilde{D})=\tilde{D}\backslash S(D)$ for $\tilde{D}\in GKD(D)$. We refer to $f_D(\tilde{D})$ as the \textbf{reduction} of $\tilde{D}$ in $GKD(D)$.
\end{definition}

\begin{example}
    Letting $D$ be the diagram of Figure~\ref{fig:Fex} \textup{(a)}, the diagram $f_D(D)$ is illustrated in Figure~\ref{fig:red}.
    \begin{figure}[H]
        \centering
        $$\scalebox{0.8}{\begin{tikzpicture}[scale=0.65]
        \node at (3.5, 1.5) {$\cdot$};
        
        \node at (2.5, 2.5) {$\cdot$};
        \node at (7.5, 2.5) {$\cdot$};

        \node at (3.5, 3.5) {$\cdot$};
        \node at (4.5, 3.5) {$\cdot$};

        \node at (2.5, 4.5) {$\cdot$};
        \node at (3.5, 4.5) {$\cdot$};

        \node at (7.5, 5.5) {$\cdot$};

        \node at (2.5, 6.5) {$\cdot$};
        
        \draw (0,7.5)--(0,0)--(8.5,0);
        
        \draw (3,1)--(4,1)--(4,2)--(3,2)--(3,1);
        
        \draw (2,2)--(3,2)--(3,3)--(2,3)--(2,2);
        \draw (7,2)--(8,2)--(8,3)--(7,3)--(7,2);

        \draw (3,3)--(5,3)--(5,4)--(3,4)--(3,3);
        \draw (4,3)--(4,4);

        \draw (2,4)--(4,4)--(4,5)--(2,5)--(2,4);
        \draw (3,4)--(3,5);

        \draw (7,5)--(8,5)--(8,6)--(7,6)--(7,5);

        \draw (2,6)--(3,6)--(3,7)--(2,7)--(2,6);
\end{tikzpicture}}$$
        \caption{$f_D(D)$}
        \label{fig:red}
    \end{figure}
\end{example}

Our main result concerning $f_D$ is as follows.

\begin{theorem}\label{thm:red}
    If $D$ is a diagram that contains no ghost cells, then $f_D:GKD(D)\to GKD(f_D(D))$ is a bijection which preserves ghost cells. Consequently, $\widehat{\mathrm{MaxG}}(D)=\widehat{\mathrm{MaxG}}(f_D(D))$.
\end{theorem}

To prove Theorem~\ref{thm:red}, we require the following proposition. Proposition~\ref{prop:comm} (a) and (b) concern properties of the set $S(D)$ of cells removed by $f_D$, while Proposition~\ref{prop:comm} (c) shows that $f_D$ commutes with the application ghost moves.

\begin{prop}\label{prop:comm}
    Let $D$ be a diagram that contains no ghost cells and $T\in GKD(D)$.
    \begin{enumerate}
        \item[$(a)$] $S(D)\subset T$.
        \item[$(b)$] If $(\widehat{r},c)\in R(T)$ and $(\tilde{r},c)\in T$ with $\tilde{r}\le \widehat{r}$, then $(\tilde{r},c)\notin S(D)$.
        \item[$(c)$] For $r>0$, we have that $\mathcal{G}(f_D(T),r)=f_D(\mathcal{G}(T,r))$, i.e., $f_D$ commutes with the application of ghost moves.
    \end{enumerate}
\end{prop}
\begin{proof}
    (a) Take $(r,c)\in S(D)$. Considering the definition of $S(D)$, there exists $c^*>c$ such that $(r,c^*)\in D$. Thus, applying Lemma~\ref{lem:block}, it follows that $(r,c)\in T$ for all $T\in GKD(D)$, as desired.

    (b) Considering the definition of $S(D)$, it suffices to show that there exists $r^*\ge \tilde{r}$ for which $(r^*,c)\in R(D)$. Assume otherwise. Let $\{r_i\}_{i=1}^n=\{r~|~(r,c)\in D,~r\ge \tilde{r}\}$; note that since $(\tilde{r},c),(\widehat{r},c)\in T$ with $\widehat{r}\ge \tilde{r}$ and $T\in GKD(D)$, it follows that $n>1$. Then for $i\in [n]$, there exists $c_i>c$ such that $(r_i,c_i)\in R(D)$. Consequently, applying Lemma~\ref{lem:block}, for all $\tilde{D}\in GKD(D)$, we have $\{r~|~(r,c)\in \tilde{D},~r\ge \tilde{r}\}=\{r_i\}_{i=1}^n$ and there exists no $r^*\ge \tilde{r}$ satisfying $(r^*,c)\in R(\tilde{D})$, which is a contradiction.

    (c) Evidently, the result holds if $\mathcal{G}(T,r)=T$, so assume that $\mathcal{G}(T,r)=T\ldownarrow^{(r,c)}_{(\widehat{r},c)}\cup\{\langle r,c\rangle\}$. Since $S(D)\subset \tilde{D}$ for all $\tilde{D}\in GKD(D)$, it follows that $(r,c),(\widehat{r},c)\notin S(D)$. Consequently, $$f_D(\mathcal{G}(T,r))=(f_D(T)\backslash\{(r,c)\})\cup \{(\widehat{r},c),\langle r,c\rangle\}.$$ Now, since $(r,c)\in R(T)$ and $(r,c)\notin S(D)$, it follows that $(r,c)\in R(f_D(T))$. Moreover, applying part (b), we have that $(\tilde{r},c)\notin S(D)$ for all $\tilde{r}<r$ such that $(\tilde{r},c)\in T$. Thus, $(r,c)\in R(f_D(T))$, $(\tilde{r},c)\in f_D(T)$ for $\widehat{r}<\tilde{r}< r$, and $(\widehat{r},c)\notin f_D(T)$ so that $$\mathcal{G}(f_D(T),r)=f_D(T)\ldownarrow^{(r,c)}_{(\widehat{r},c)}\cup\{\langle r,c\rangle\}=(f_D(T)\backslash\{(r,c)\})\cup \{(\widehat{r},c),\langle r,c\rangle\}=f_D(\mathcal{G}(T,r)),$$ as desired.
\end{proof}

We are now in a position to prove Theorem~\ref{thm:red}.

\begin{proof}[Proof of Theorem~\ref{thm:red}]
First, we show that if $T\in GKD(D)$, then $f_D(T)\in GKD(f_D(D))$. Evidently, $f_D(D)\in GKD(f_D(D))$. Take $T\in GKD(D)$ such that $T\neq D$. By definition, there exists a sequence of rows $\{r_i\}_{i=1}^n$ such that if $D_0=D$ and $D_i=\mathcal{G}(D_{i-1},r_i)$ for $i\in [n]$, then $D_n=T$. Applying Proposition~\ref{prop:comm} (c), we have that $$f_D(D_i)=f_D(\mathcal{G}(D_{i-1},r_i))=\mathcal{G}(f_D(D_{i-1}),r_i)$$ for $i\in[n]$. Thus, since $f_D(D_0)=f_D(D)$, there exists a sequence of ghost moves which can be applied to form $f_D(D_n)=f_D(T)$ from $f_D(D)$, i.e., $f_D(T)\in GKD(f_D(D))$. Consequently, $f_D:GKD(D)\to GKD(f_D(D))$.

Now, since $S(D)\subset \tilde{D}$ for all $\tilde{D}\in GKD(D)$ and $S(D)$ contains no ghost cells, it follows that $G(\tilde{D})=G(\tilde{D}\backslash S(D))=G(f_D(\tilde{D}))$ for all $\tilde{D}\in GKD(D)$, i.e., $f_D$ is ghost cell preserving. To see that $f_D$ is one-to-one, take $D_1, D_2 \in GKD(D)$ such that $f_D(D_1) = f_D(D_2)$. Then $D_1 \backslash S(D) = D_2 \backslash S(D)$. Now, since $S(D) \subset D_1,D_2$, we have that
\begin{align*}
    D_1 & = (D_1 \backslash S(D)) \cup S(D)\\
    & = f_D(D_1) \cup S(D)\\
    & = f_D(D_2) \cup S(D)\\
    & = (D_2 \backslash S(D)) \cup S(D)\\
    & = D_2;
\end{align*}
that is, $D_1=D_2$ so that $f_D$ is one-to-one. Finally, it remains to show that $f_D$ is onto. For a contradiction, assume that there exists $T \in GKD(f_D(D))$ which is not in the image of $f_D$. Since $f_D(D)$ is in the image of $f_D$, it follows that there must diagrams $T_1,T_2\in GKD(D)$ such that $T_1$ is in the image of $f_D$, $T_2=\mathcal{G}(T_1,r)$ for some $r>1$, and $T_2$ is not in the image of $f_D$. Assume that $\widehat{T}_1\in GKD(D)$ satisfies $f_D(\widehat{T}_1)=T_1$. Then applying Proposition~\ref{prop:comm} (c), we have that $$T_2=\mathcal{G}(T_1,r)=\mathcal{G}(f_D(\widehat{T}_1),r)=f_D(\mathcal{G}(\widehat{T}_1,r));$$ but this implies that $f_D$ maps $\mathcal{G}(\widehat{T}_1,r)$ to $T_2$, a contradiction. Thus, $f_D$ is onto and the result follows.
\end{proof}

\section{Future Work}\label{sec:epi}

In this paper, we study the combinatorial puzzle that arises from the definition of Lascoux polynomials in terms of diagrams and $K$-Kohnert moves \textbf{\cite{Pan1,Pan2,KKohnert}}. Given an arbitrary diagram $D$, the object of this puzzle is to apply a sequence of $K$-Kohnert moves to create a diagram $T$ from $D$ which contains the maximum possible number of ghost cells. Our goal here was to establish means by which one can determine if they have solved such a puzzle; that is, means of computing the number of ghost cells contained in a solution to the puzzle associated with the diagram $D$, denoted $\mathrm{MaxG}(D)$. As noted earlier, we are not the first to consider this question. In \textbf{\cite{Pan2}}, for a key diagram $D$, the authors establish means of computing $\mathrm{MaxG}(D)$ and provide a sequence of $K$-Kohnert moves which, when applied to $D$, produce a diagram $T\in KKD(D)$ containing $\mathrm{MaxG}(D)$ many ghost cells, i.e., a solution to the associated puzzle. Here, we show that the method for computing $\mathrm{MaxG}(D)$ when $D$ is a key diagram introduced in \textbf{\cite{Pan2}} can be applied either directly or with slight modification to a larger class of diagrams, including the skew diagrams and certain lock diagrams of \textbf{\cite{KP1}}. In addition, we consider the limits of taking a greedy approach to such puzzles, establishing means of computing the value analogous to $\mathrm{MaxG}(D)$ when proceeding greedily.

Aside from continuing to extend the results here and in \textbf{\cite{Pan2}}, with the goal of introducing means of computing $\mathrm{MaxG}(D)$ and solving this puzzle for arbitrary diagrams $D$, there exist other interesting directions for future work. For example, one could construct and study 2-person games arising from the puzzle. Two possible ways for extending the puzzle to a 2-player game are
\begin{itemize}
    \item[(1)] having players alternate making $K$-Kohnert moves under normal (resp. mis\'ere) play; or
    \item[(2)] having players alternate making $K$-Kohnert moves while keeping track of the number of ghost cells created on their turns, the player having created the most ghost cells when no moves are left being the winner.
\end{itemize}
For another possible direction of future work, one could replace the collection of $K$-Kohnert moves used here with the newly introduced set of moves in \textbf{\cite{Groth}}. In \textbf{\cite{Groth}} it is conjectured that the new moves introduced within allow for a definition of Grothendieck polynomials analogous to that described here for Lascoux polynomails. 


\printbibliography

\end{document}